\documentclass[12pt, reqno]{amsart}
\usepackage{amsfonts, amsmath, amsthm, amscd, amsfonts, amssymb, graphicx, color}
\usepackage{tikz-cd}
\usepackage[bookmarksnumbered, colorlinks, plainpages]{hyperref}
\newtheorem{theorem}{Theorem}[section]
\newtheorem{lemma}[theorem]{Lemma}
\newtheorem{proposition}[theorem]{Proposition}

\newtheorem{corollary}[theorem]{Corollary}
\newtheorem{definition}[theorem]{Definition}
\newtheorem{example}[theorem]{Example}
\newtheorem{remark}[theorem]{Remark}

\usepackage[bookmarksnumbered, colorlinks, plainpages]{hyperref}
\hypersetup{colorlinks=true,linkcolor=blue, anchorcolor=green, citecolor=cyan, urlcolor=red, filecolor=magenta, pdftoolbar=true}

\usepackage{comment}
\excludecomment{mysection}

\begin{document}
	
\title[$K_1$ and $K$-groups of absolute matrix order unit spaces]{$K_1$ and $K$-groups of absolute Matrix order unit spaces}
\author{Amit Kumar}
	
\address{Discipline of Mathematics, School of Basic Sciences, Indian Institute of Technology Bhubaneswar, Argul, Bhubaneswar, Pin - 752050, Odisha (State), India.}

\email{\textcolor[rgb]{0.00,0.00,0.84}{amit231291280895@gmail.com}}

	

\begin{abstract}
In this paper, we describe the Grothendieck groups $K_1(V)$ and $K(V)$ of an absolute matrix order unit space $V$ for unitary and partial unitary elements respectively. For this purpose, we study some basic properties of unitary and partial unitary elements, and define their path homotopy equivalence. The construction of $K(V)$ follows in a almost similar manner as that of $K_1(V).$ We prove that $K_1(V)$ and $K(V)$ are ordered abelian groups. We also prove that $K_1(V)$ and $K(V)$ are functors from the category of absolute matrix order unit spaces with morphisms as unital completely $\vert \cdot \vert$-preserving maps to the category of ordered abelian groups. Later, we show that under certain conditions, quotient of $K(V)$ is isomorphic to the direct sum of $K_0(V)$ and $K_1(V),$ where $K_0(V)$ is the Grothendieck group for order projections.
\end{abstract}

\thanks{The author was financially supported by the Institute Post-doctoral Fellowship of IIT Bhubaneswar, India.}

\maketitle

\section{Introduction}

Operator $K$-theory dwells in the heart of the area of non-commutative topology. It is an active research area and a much used tool for the study of C$^*$-algebras. $K$-theory  was introduced in C$^*$-algebra theory in the early 1970s through some specific applications. However, the germs of $K$-theory are found in the Grothendieck's work in algebraic geometry. $K$-theory was developed by Atiyah and Hirzebruch in the 1960s based on Grothendieck's this work. It plays a major role in structure theory of C$^*$-algebras. Due to its applications and appeal, operator $K$-theory has been central point of attraction for a wide range of mathematicians. For such and more informations, see \cite{B98, B06, RLL00, WO93} and references therein. 

Order structure is one of the basic component of C$^*$-algebra theory. In the light of Gelfand-Neumark theorem \cite{GN} and Kakutani theorem \cite{Kak}, we see that self-adjoint part of a commutative C$^*$-algebra forms a Banach lattice. However, as a contrast, Kadison's anti-lattice theorem \cite{Kad51} explains that the self-adjoint part of a non-commutative C$^*$-algebra is not a vector lattice. Therefore, the order structure of a C$^*$-algebra becomes an interesting area for study. The corresponding theory gives a birth to study of orderd vector spaces which do not have any vector lattice structure. Being motivated by works of Kadison, Effros, Str\o{}mer and Pedersen, and many others (see \cite{GKP} and references therein) which describe various aspects of order structure of C$^*$-algebras and encourage to expect a `non-commutative vector lattice' or a `near lattice' structure in it, Karn started working in this direction (see \cite{K10, K14, K16, K18} etc.).

The significant contribution in this direction begins with Kadison's functional representation of operator (C$^*$-) algebras \cite{RVK}. By this represenation, the self-adjoint part of every unital C$^*$-algebra $A$ can be expressed as the space of continuous affine functions on the state space $S(A)$ of $A.$ It becomes clear that the order structure of a C$^*$-algebra is rich with many properties. For more, we refer to \cite{LA, LA73, A74, B54, B56, E64, JE64, KFN} and references therein.

In few years, the author, in collaboration with Karn, has been working on the order theoretic aspects of C$^*$-algebras (see \cite{K19,PI19,KO21}). Karn is already working in this direction, in the collaboration with several others. In \cite{K18}, he introduced the notion of absolute order unit spaces. Self-adjoiint parts of unital C$^*$-algebras and unital $M$-spaces are examples of absolute order unit spaces. More generally, unital $JB$-algebras are examples of absolute order unit spaces. It is shown that under an additional condition (see \cite[Theorem 4.12]{K16}) absolute order unit spaces turn out to be vector lattices. Therefore an absolute order unit space may be termed as a `non-commutative vector lattice'.

In \cite{K19}, Karn with the author introduced and studied the notion of absolute matrix order unit space which is a matricial version of absolute order unit space. We also studied absolute value preseving (completely absolute value preseving) maps between two absolute order unit spaces (absolute matrix order unit spaces). We found that unital bijective absolute value preseving (completely absolute value preseving) maps between two absolute order unit spaces (absolute matrix order unit spaces) are isometries (complete isometries) \cite[Theorems 3.3 and 4.6]{K19}.

The notion of order projections in absolute order unit spaces has been introduced and studied by Karn in \cite{K18} and it generalizes order theoretically projections in C$^*$-algebras. In \cite{PI19} (continuation of \cite{K19}), Karn with the author generalized the notion of order projections in absolute order unit spaces to order projections in absolute matrix order unit spaces. We introduced and studied the notion of partial isometries in absolute matrix order unit spaces to study comparison of order projections. We also studied some variants of order projections in absolute matrix order unit spaces. 

Furher, in the continuation of \cite{K19} and \cite{PI19}, Karn with the author has worked in this direction in \cite{KO21}. In \cite{KO21}, we initiate a study of $K$-theory in absolute matrix order unit spaces. We studied direct limit $M_\infty(V)$ and order projections in $M_\infty(V)$ to describe Grothendieck group $K_0(V)$ of an absolute matrix order unit space $V$ following the same path as Grothendieck group $K_0(A)$ of a unital C$^*$-algebra $A$ is defined via the class of projections $\mathcal{P}_{\infty}(A)$ corresponding to the inductive (direct) limit $M_{\infty}(A)$ of the family of C$^*$-algebras $\lbrace M_n(A) \rbrace$. Later, we proved that $K_0(V)$ is a functor from the category of absolute matrix order unit spaces with morphisms as unital completely absolute value preserving maps to the category of ordered abelian groups \cite[Corollary 5.3]{KO21}. 

Let's recall the notion of order projections in $M_\infty(V)$ introduced in \cite{KO21}. Let $\mathfrak{p} \in M_\infty(V)$ and let $o(\mathfrak{p})$ be the order of $\mathfrak{p}$. Then $\mathfrak{p}$ is said to be order projection, if $\vert \mathfrak{e}^n-2\mathfrak{p}\vert=\mathfrak{e}^n$ for all $n\in \mathbb{N}$ such that $n\geq o(\mathfrak{p}).$ For an absolute matrix order unit space $V,$ we denote the set of all order projections, unitaries and partial unitaries in $M_n(V)$ by $\mathcal{OP}_n(V), \mathcal{U}_n(V)$ and $\mathcal{PU}_n(V)$ respectively, for any $n\in \mathbb{N}$ (for definitions see \cite[3.1]{PI19}). In \cite{KO21}, it is discussed that under the identification $M_\infty(V)=\displaystyle\bigcup_{n=1}^\infty M_n(V),$ the corresponding set of order projections in $M_\infty(V)$ is identified with $\mathcal{OP}_\infty(V)=\displaystyle\bigcup_{n=1}^\infty \mathcal{OP}_n(V).$ Similarly, we can also define the notion of unitaries and partial unitaries in $M_\infty(V)$ in the following way: Let $\nu \in M_\infty(V).$ Then $\nu$ is said to be unitary, if $\vert \nu\vert=\vert \nu^*\vert = \mathfrak{e}^{o(\nu)}$ and it is said to be partial unitary, if $\vert \nu\vert=\vert \nu^*\vert$ is an order projection. Again, under the identification $M_\infty(V)=\displaystyle\bigcup_{n=1}^\infty M_n(V),$ the corresponding set of unitaries and partial unitaries in $M_\infty(V)$ can be identified with $\mathcal{U}_\infty(V)=\displaystyle\bigcup_{n=1}^\infty \mathcal{U}_n(V)$ and $\mathcal{PU}_\infty(V)=\displaystyle\bigcup_{n=1}^\infty\mathcal{PU}_n(V)$ respectively.

In this paper, we continue our work from \cite{KO21}. We study some basic properties of unitaries and partial unitaries in an absolute matrix order unit space (Propositions \ref{30} and \ref{65}). We also study path homotopy equivalences of unitary and partial unitaries (Propositions \ref{33}, \ref{32}, \ref{67}, \ref{68} and Lemmas \ref{34}, \ref{69}). By help of path homotopy, we induce new equivalence relations for unitary and partial unitaries, and study their properties (Corollaries \ref{35}, \ref{62}, \ref{70}, \ref{73} and Propositions \ref{40}, \ref{61}, \ref{63}, \ref{71}, \ref{72}, \ref{74}). By such equivalence relations, we describe $K_1(V)$ and $K(V)$ the Grothendieck groups of an absoute matrix order unit space $V$ for unitary and partial unitaries elements respectively (Theorems \ref{64} and \ref{75}), which are generalizations of $K_1(A)$ and $K(A)$ of a C$^*$-algebra $A$ for unitary and partial unitaries. We show that $K_1(V)$ and $K(V)$ are ordered abelian groups (Theorems \ref{94} and \ref{84}). We also prove that $K_1(V)$ and $K(V)$ are functors from the category of absolute matrix order unit spaces with morphisms as unital completely $\vert \cdot \vert$-preserving maps to the category of ordered abelian groups (Corollaries \ref{80} and \ref{83}). Later, we find a relation of $K_1(V)$ and $K(V)$ with $K_0(V) $ (Theorem \ref{85}).

\section{$K_0$-group corresponding to an absolute matrix order unit space}
In this section, we recall the $K_0$-group of an absolute matrix order unit space that has been described by Karn and the author in \cite{KO21}. For that, we need to recall some preliminaries. We start this section with the following matricial notions:

Let $V$ be a complex vector space. We denote by $M_{m,n}(V)$ the vector space of all the $m \times n$ matrices $v=[v_{i,j}]$ with entries $v_{i,j} \in V$ and by $M_{m,n}$ the vector space of all the $m \times n$ matrices $a = [a_{i,j}]$ with entries $a_{i,j} \in \mathbb{C}$. We write $0_{m,n}$ for zero element in $M_{m,n}(V)$. For $m=n$, we write $0_{m,n} = 0_n$. We define $av = \begin{bmatrix} \displaystyle \sum_{k=1}^m a_{i,k}v_{k,j}\end{bmatrix}$ and $vb = \begin{bmatrix} \displaystyle \sum_{k=1}^n v_{i,k} b_{k,j} \end{bmatrix}$ for $a \in M_{r,m}, v \in M_{m,n}(V)$ and $b \in M_{n,s}$. We write 

\begin{center}
$v \oplus w = \begin{bmatrix} v & 0 \\ 0 & w\end{bmatrix}$ for $v \in M_{m,n}(V),w \in M_{r,s}(V).$
\end{center}

Here $0$ denotes suitable rectangular matrix of zero entries from $V$ \cite{ZJR88}. 

\subsection{Matrix Order Unit Spaces}
In this subsection, we recall some matrix theory of ordered vector spaces. For this theory, we refer to \cite{ZJR88, CE77, ER88, KV97} and references therein.

A complex vector space with an involution is called a \emph{$\ast$-vector space}. We write $V_{sa} = \lbrace v \in V: v = v^*\rbrace$. Then $V_{sa}$ is a real vector space \cite{CE77}. 

Next, we recall the definition of a matrix ordered space.

\begin{definition}[\cite{CE77}]
A \emph{matrix ordered space} is a $*$-vector space $V$ together with a sequence $\lbrace M_n(V)^+ \rbrace$ with $M_n(V)^+ \subset M_n(V)_{sa}$ for each $n \in \mathbb{N}$ satisfying the following conditions: 
\begin{enumerate}
	\item[(a)] $(M_n(V)_{sa}, M_n(V)^+)$ is a real ordered vector space, for each $n \in \mathbb{N}$; and  
	\item[(b)] $a^* v a \in M_m(V)^+$ for all $v \in M_n(V)^+$, $a \in M_{n,m}$ and $n ,m \in \mathbb{N}$. 
\end{enumerate} 
It is denoted by $(V, \lbrace M_n(V)^+ \rbrace)$. If, in addition, $e \in V^+$ is an order unit in $V_{sa}$ such that $V^+$ is proper and $M_n(V)^+$ is Archimedean for all $n \in \mathbb{N}$, then $V$ is called a \emph{matrix order unit space} and is denoted by $(V, \lbrace M_n(V)^+ \rbrace, e)$.
\end{definition}

The first part of next result elaborates that each level in a matrix ordered space is proper and generating provided it is proper and generating at ground level respectively, and the second part elaborates that a matrix ordered space has order unit at each level provided it has order unit at ground level.

\begin{proposition}\label{81}
Let $(V, \lbrace M_n(V)^+\rbrace)$ be a matrix ordered space. 
\begin{enumerate}
\item[(1)]$($\cite{KV97}, {\it Proposition} 1.8$).$
\begin{enumerate}
\item[(a)] If $V^+$ is proper, then $M_n(V)^+$ is proper for all $n \in \mathbb{N}$.
\item[(b)] If $V^+$ is generating, then $M_n(V)^+$ is generating for all $n \in \mathbb{N}$.
\end{enumerate}
\item[(2)] $($\cite{KV97}, Lemma 2.6$).$ If $e \in V^+$ is an order unit for $V_{sa}$. Then $e^n$ is an order unit for $M_n(V)_{sa}$ for all $n \in \mathbb{N}$ ( where $e^n : = e \oplus  \cdots \oplus e \in M_n(V)$).
\end{enumerate}
\end{proposition}

In fact, matrix ordered spaces are normed spaces that is clear from the following result:

\begin{theorem}[\cite{CE77}]\label{1} 
Let $(V, e)$ be a matrix order unit space. Then 
\begin{enumerate}
\item[(1)] The order unit determines a matrix norm on $V$ in the following way: 
$$\Vert v \Vert_n := \inf \lbrace k > 0: \begin{bmatrix} k e^n & v \\ v^* & k e^n \end{bmatrix} \in M_{2n}(V)^+ \rbrace \quad \textrm{for} ~ v \in M_n(V).$$ 

\item[(2)] $\left \Vert \begin{bmatrix} u & 0\\ 0 & v \end{bmatrix} \right \Vert_{m+n}=\max \lbrace \Vert u\Vert_m, \Vert v\Vert_n \rbrace$ for all $u\in M_m(V)$ and $v\in M_n(V).$

\item[(3)] $\Vert \alpha v\beta \Vert_n \leq \Vert \alpha \Vert \Vert v\Vert_n \Vert \beta\Vert$ for all $\alpha,\beta \in M_n.$ 

\item[(4)]$\Vert \alpha v \beta \Vert_n=\Vert v\Vert_n$ for all $\alpha,\beta \in M_n$ such that $\alpha^*\alpha=I_n$ and $\beta \beta^*=I_n$ (i.e. for $\alpha$ isometry and $\beta$ co-isometry).

\item[(5)] $\Vert v\Vert_n =\Vert v^*\Vert_n=\left\Vert \begin{bmatrix} 0 & v\\ v^* & 0\end{bmatrix}\right \Vert_{2n}.$
\end{enumerate}
In other words, $(V,\lbrace\Vert \cdot\Vert_n\rbrace)$ is a $L^\infty$-matricially normed space \cite{ZJR88}.
\end{theorem}

In \cite{K10,K14,K16,K18,K19,PI19}, some variants of orthoganity in ordered normed spaces have been introduced and studied in terms of norm by Karn. We recall some of them here:

Let $u, v \in M_n(V)^+$. We say that $u$ is $\infty$-\emph{orthogonal}, (we write $u \perp_{\infty} v$), if $\Vert u + k v \Vert_n = \max \lbrace \Vert u \Vert, \Vert k v \Vert_n \rbrace$ for all $k \in \mathbb{R}$ \cite[Definition 3.6]{K18}. Recall that for $u, v \in M_n(V)^+ \setminus \lbrace 0 \rbrace$, we have $u \perp_{\infty} v$ if and only if $\Vert \Vert u \Vert_n^{- 1} u + \Vert v \Vert_n^{- 1} v \Vert_n = 1$ \cite[Theorem 3.3]{K14}. For $u, v \in M_n(V)^+ \setminus \lbrace 0 \rbrace$, we say that $u \perp_{\infty}^a v$, if $u_1 \perp_{\infty} v_1$ whenever $0 \le u_1 \le u$ and $0 \le v_1 \le v$. It was proved in \cite{K16, K18} that if $A$ is a C$^*$-algebra and if $a, b \in A^+ \setminus \lbrace 0 \rbrace$, then $a \perp_{\infty}^a b$ if and only if $a b = 0$ (that is, $a$ is algebraically orthogonal to $b$).

\subsection{Absolute Matrix Order Unit Spaces}
In this subsection, we briefly recall absolutely matrix ordered spaces and absolute matrix order unit spaces which are matricial versions are absolutely ordered spaces and absolute order unit spaces respectively. The theory of absolutely ordered spaces and absolute order unit spaces have been introduced and studied by Karn in \cite{K18}, and the theory of absolutely matrix ordered spaces and absolute matrix order unit spaces have been introduced and studied by Karn and the author in \cite{K19}. We begin with the definition of absolutely matrix ordered spaces.

\begin{definition}[\cite{K19}, Definition 4.1]\label{152}
Let $(V, \lbrace \ M_n(V)^+ \rbrace)$ be a matrix ordered space and assume that  $\vert\cdot\vert_{m,n}: M_{m,n}(V) \to M_n(V)^+$ for $m, n \in \mathbb{N}$. Let us write $\vert\cdot\vert_{n,n} = \vert\cdot\vert_n$ for every $n \in \mathbb{N}$. Then $\left(V, \lbrace M_n(V)^+ \rbrace, \lbrace \vert\cdot\vert_{m,n} \rbrace \right)$ is called an \emph{absolutely matrix ordered space}, if it satisfies the following conditions: 
\begin{enumerate}
\item[$1.$] For all $n \in \mathbb{N}$, $(M_n(V)_{sa}, M_n(V)^+, \vert\cdot\vert_n)$ is an absolutely ordered space;
\item[$2.$] For $v \in M_{m,n}(V), \alpha \in M_{r,m}$ and $\beta \in M_{n,s},$ we have
$$\vert \alpha v \beta \vert_{r,s} \leq \| \alpha \| \vert \vert v \vert_{m,n} \beta \vert_{n,s};$$
\item[$3.$] For $v \in M_{m,n}(V)$ and $w \in M_{r,s}(V),$ we have
$$\vert v \oplus w\vert_{m+r,n+s} = \vert v \vert_{m,n} \oplus \vert w \vert_{r,s}.$$
\end{enumerate} 
\end{definition}

\begin{example}[\cite{K19}, Example 4.4]
Let $A$ be a $C^*$-algebra. Then, for each $n\in \mathbb{N},M_n(A)$ is a $C^*$-algebra. If $M_n(A)^+$ denotes the set of all the positive elements in $M_n(A),$ then $(A,\lbrace M_n(A)^+\rbrace)$ is a matrix ordered space. For $m,n \in \mathbb{N},$ we define $\vert \cdot \vert_{m,n}:M_{m,n}(A)\to M_n(A)^+$ given by $\vert a\vert_{m,n}= (a^*a)^{\frac{1}{2}}$ for all $a\in M_{m,n}(A).$ Then $(A,\lbrace M_n(A)^+\rbrace,\lbrace \vert \cdot \vert_{m,n}\rbrace)$ is an absolutely matrix ordered space. 
\end{example}

The following two results describes some basic properties of absolutely matrix ordered spaces.

\begin{proposition}[\cite{K19}, Proposition 4.2]\label{91}
Let $(V, \lbrace M_n(V)^+ \rbrace, \lbrace \vert\cdot\vert_{m,n} \rbrace)$ be an absolutely matrix ordered space.  
\begin{enumerate}
\item[(1)] If $\alpha \in M_{r,m}$ is an isometry i.e. $\alpha^* \alpha = I_m,$ then $\vert \alpha v \vert_{r,n} = \vert v \vert_{m,n}$ for any $v \in M_{m,n}(V).$
\item[(2)] If $v \in M_{m,n}(V),$ then $\left\vert \begin{bmatrix} 0_m & v \\ v^* & 0_n \end{bmatrix} \right\vert_{m+n} = \vert v^* \vert_{n,m} \oplus \vert v \vert_{m,n}.$
\item[(3)] $\begin{bmatrix} \vert v^* \vert_{n,m} & v \\ v^* & \vert v \vert_{m,n} \end{bmatrix} \in M_{m+n}(V)^+$ for any $v \in M_{m,n}(V).$
\item[(4)] $\vert v \vert_{m,n} = \left\vert \begin{bmatrix} v \\ 0 \end{bmatrix} \right\vert_{m+r,n}$ for any $v \in M_{m,n}(V)$ and $r \in \mathbb{N}.$
\item[(5)] $\vert v \vert_{m,n} \oplus 0_s = \left\vert \begin{bmatrix} v & 0 \end{bmatrix} \right\vert_{m,n+s}$ for any $v \in M_{m,n}(V)$ and $s \in \mathbb{N}.$
\end{enumerate} 
\end{proposition}

\begin{lemma}[\cite{PI19}, Lemma 3.1]\label{95} 
Let $V$ be an absolutely matrix ordered space and let $n\in \mathbb{N}.$ Then $\vert \alpha^* v \alpha \vert_n = \alpha^* \vert v\vert_n \alpha,$ if $v \in M_n(V)$ and $\alpha \in M_n$ is a unitary. 
\end{lemma}

Next, we recall the another variant of orthogonality for general elements in absolutely matrix ordered spaces in terms of absolute value, introduced and studied by Karn and the author in \cite{PI19}. This orthogonaly for positive elements in absolutely ordered spaces has been introduced and studied earlier by Karn in \cite{K18}.

\begin{definition}[\cite{PI19}, Definitions 2.2 and 2.4]	
Let $\left(V, \lbrace M_n(V)^+ \rbrace, \lbrace \vert\cdot\vert_{m,n} \rbrace \right)$ be an \emph{absolutely matrix ordered space}. Let $u,v \in M_n(V)^+$ for some $n\in \mathbb{N}.$ We say that $u$ is \emph{orthogonal} to $v,$ (we write it, $u\perp v$), if $\vert u-v\vert_n=u+v.$ Now, if $u,v \in M_n(V)_{sa},$ we say that $u$ is \emph{orthogonal} to $v,$ (we still write it, $u\perp v$), if $\vert u\vert_n \perp \vert v\vert_n.$ Further, if $u, v \in M_{m,n}(V)$ for some $m,n \in \mathbb{N}$, we say that $u$ is \emph{orthogonal} to $v$, (we continue to write, $u \perp v$), if $\begin{bmatrix} 0 & u \\ u^* & 0 \end{bmatrix} \perp \begin{bmatrix} 0 & v \\ v^* & 0 \end{bmatrix}$ in $M_{m+n}(V)_{sa}$.
\end{definition}

The next result associates orthogonality of general elements in absolutely matrix ordered spaces to orthogonality of positive elements. 

\begin{remark}[\cite{PI19}, Remark 2.2(2)]\label{92}
Let $V$ be an absolutely matrix ordered space and let $u, v \in M_{m,n}(V)$ for some $m,n \in \mathbb{N}.$ Then $u\perp v$ if and only if $\vert u\vert_{m,n} \perp \vert v\vert_{m,n}$ and $\vert u^*\vert_{n,m} \perp \vert v^*\vert_{n,m}.$
\end{remark}

The following result tells that absolute values are additive on orthogonal elements.

\begin{proposition}[\cite{PI19}, Proposition 2.3(1)]\label{88}
	Let $V$ be an absolutely matrix ordered space and let $u,v\in M_{m,n}(V)$. Then $u\perp v$ if and only if, $\vert u\pm v\vert_{m,n} = \vert u\vert_{m,n} +\vert v\vert_{m,n}$ and $\vert u^*\pm v^*\vert_{n,m} = \vert u^*\vert_{n,m} +\vert v^*\vert_{n,m}.$
\end{proposition}

Next, we prove that orthogonality in absolutely matrix ordered spaces is invariant under scalar multiplication.

\begin{proposition}\label{90}
Let $V$ be an absolutely matrix ordered space and let $u,v\in M_{m,n}(V)$ such that $u\perp v.$ Then $\alpha u \perp \beta v$ for all $\alpha,\beta \in \mathbb{C}.$
\end{proposition}

\begin{proof}
Since $u\perp v,$ we have $\left \vert \begin{bmatrix} 0 & u \\ u^* & 0 \end{bmatrix}\right \vert_{m+n} \perp \left \vert\begin{bmatrix} 0 & v \\ v^* & 0 \end{bmatrix}\right \vert_{m+n}.$ Let $\alpha,\beta \in \mathbb{C}.$ If $\alpha=0$ or $\beta=0,$ then  $\alpha u \perp \beta v.$ Now, assume $\alpha \neq 0,\beta \neq 0.$ Thus $\alpha=\vert \alpha\vert e^{i\zeta_1}$ and $\beta=\vert \beta\vert e^{i\zeta_2}$ for some $\zeta_1,\zeta_2 \in \mathbb{R}.$ Then $\begin{bmatrix} e^{i\zeta_1} I_m & 0 \\ 0 & e^{-i\zeta_1} I_n \end{bmatrix}$ and $\begin{bmatrix} e^{i\zeta_2} I_m & 0 \\ 0 & e^{-i\zeta_2} I_n \end{bmatrix}$ are unitary matrices such that $\begin{bmatrix} e^{i\zeta_1} I_m & 0 \\ 0 & e^{-i\zeta_1} I_n \end{bmatrix}\begin{bmatrix} 0 & u \\ u^* & 0 \end{bmatrix}=\begin{bmatrix} 0 & e^{i\zeta_1} u\\ (e^{i\zeta_1}u)^* & 0\end{bmatrix}$ and $\begin{bmatrix} e^{i\zeta_2} I_m & 0 \\ 0 & e^{-i\zeta_2} I_n \end{bmatrix}\begin{bmatrix} 0 & v \\ v^* & 0 \end{bmatrix}=\begin{bmatrix} 0 &  e^{i\zeta_2}v\\ (e^{i\zeta_2} v)^* & 0\end{bmatrix}.$ By Proposition \ref{91}(1), we get that 
$\left \vert \begin{bmatrix} 0 & e^{i\zeta_1} u\\ (e^{i\zeta_1}u)^* & 0\end{bmatrix} \right \vert_{m+n}=\left \vert \begin{bmatrix} 0 & u \\ u^* & 0 \end{bmatrix} \right \vert_{m+n} \perp \left \vert \begin{bmatrix} 0 & v \\ v^* & 0 \end{bmatrix} \right \vert_{m+n} = \left \vert \begin{bmatrix} 0 &  e^{i\zeta_2}v\\ (e^{i\zeta_2} v)^* & 0 \end{bmatrix} \right \vert_{m+n}.$ 

Next, either $\vert \alpha\vert \leq \vert\beta\vert$ or $\vert \beta\vert \leq \vert\alpha\vert.$ Assume that $\vert \alpha\vert \leq \vert\beta\vert.$ Then, by \cite[Definition 3.4(3)]{K18}, we have that $\vert \beta \vert \left \vert \begin{bmatrix} 0 & e^{i\zeta_1} u\\ (e^{i\zeta_1}u)^* & 0\end{bmatrix} \right \vert_{m+n} \perp \vert \beta\vert \left \vert \begin{bmatrix} 0 &  e^{i\zeta_2}v\\ (e^{i\zeta_2} v)^* & 0 \end{bmatrix} \right \vert_{m+n}.$ Since $\vert \alpha \vert \left \vert \begin{bmatrix} 0 & e^{i\zeta_1} u\\ (e^{i\zeta_1}u)^* & 0\end{bmatrix} \right \vert_{m+n} \leq \vert \beta \vert \left \vert \begin{bmatrix} 0 & e^{i\zeta_1} u\\ (e^{i\zeta_1}u)^* & 0\end{bmatrix} \right \vert_{m+n},$ again by \cite[Definition 3.4(4)]{K18}, we conclude that $\left \vert \begin{bmatrix} 0 & \alpha u\\ (\alpha u)^* & 0\end{bmatrix} \right \vert_{m+n}= \vert \alpha\vert \left \vert \begin{bmatrix} 0 & e^{i\zeta_1} u\\ (e^{i\zeta_1}u)^* & 0\end{bmatrix} \right \vert_{m+n} \perp \vert \beta\vert \left \vert \begin{bmatrix} 0 &  e^{i\zeta_2}v\\ (e^{i\zeta_2} v)^* & 0 \end{bmatrix} \right \vert_{m+n}=\left \vert \begin{bmatrix} 0 & \beta v\\ (\beta v)^* & 0\end{bmatrix} \right \vert_{m+n}$ so that $\alpha u \perp \beta v.$ Similarly, the case $\vert \beta\vert \leq \vert\alpha\vert$ can be taken care.
\end{proof}

Finally, we bind up this subsection by recalling absolute matrix order unit spaces.

\begin{definition}[\cite{K19}, Definition 4.3]
Let $(V, \lbrace M_n(V)^+ \rbrace, e)$ be a matrix order unit space such that 
\begin{enumerate}
\item[(a)] $\left(V, \lbrace M_n(V)^+ \rbrace, \lbrace \vert \cdot \vert_{m,n} \rbrace \right)$ is an absolutely matrix ordered space; and
\item[(b)]$\perp = \perp_{\infty}^a$ on $M_n(V)^+$ for all $n \in \mathbb{N}.$ 
\end{enumerate}
Then $(V, \lbrace M_n(V)^+ \rbrace, \lbrace \vert\cdot\vert_{m,n} \rbrace, e)$ is called an \emph{absolute matrix order unit space}. 
\end{definition}

\begin{example}[\cite{K19}, Example 4.4]
For any $C^*$-algebra $A,$ we have $\perp = \perp_\infty^a$ on $M_n(A)^+$ for all $n \in \mathbb{N}.$ Therefore, every unital $C^*$-algebra is an absolute matrix order unit space.
\end{example}

\subsection{Order Projections in Absolute Matrix Order Unit Spaces}

The notion of order projections has been introduced and studied in absolute order unit spaces by Karn in \cite{K18}. It generalizes the notion of projections in $C^*$-algebras. In the collaboration with Karn, author has generalized the notion of order projections in absolute order unit spaces to order projections in absolute matrix order unit spaces \cite{PI19}. In \cite{PI19}, we have also introduced and studied the notion of partial isometries in absolute matrix order unit spaces which generalizes the notion of partial isometries in $C^*$-algebras. For details, see \cite{K18, PI19}.

In this subsection, we recall the notions of order projections and partial isometries in absolute matrix order unit spaces.

\begin{definition}[\cite{PI19}, Definition 3.1]
Let $(V,e)$ be an absolute matrix order unit space and let $v \in M_n(V)$ for some $n\in \mathbb{N}.$ Then $v$ is said to be \emph{order projection}, if $v^* = v$ and $\vert 2 v - e^n \vert_n = e^n.$ Now, let $v \in M_{m,n}(V)$ for some $m, n \in \mathbb{N}.$ Then $v$ is said to be \emph{partial isometry}, if $\vert v \vert_{m,n}$ and $\vert v^* \vert_{n,m}$ are order projections.
\end{definition}

The set of all the order projections in $M_n(V)$ will be denoted by $\mathcal{OP}_n(V)$ and the set of all partial isometries in $M_{m,n}(V)$ will be denoted by $\mathcal{PI}_{m,n}(V).$ For $m = n$, we write $\mathcal{PI}_{m,n}(V) = \mathcal{PI}_n(V)$. For $n = 1$, we shall write $\mathcal{PI}(V)$ for $\mathcal{PI}_1(V)$ and $\mathcal{OP}(V)$ for $\mathcal{OP}_1(V)$. 

First, let's recall the following property of order projections:

\begin{proposition}[\cite{PI19}, Proposition 3.2] \label{3}
Let $V$ be an absolute matrix order unit space and let $m,n\in \mathbb{N}.$ Then $p\in \mathcal{OP}_m(V), q\in \mathcal{OP}_n(V)$ if and only if  $p\oplus q\in \mathcal{OP}_{m+n}(V)$
\end{proposition}

The following result describes how othogonality helps partial isometries for their sum to be again a partial isometry. 

\begin{proposition}[\cite{PI19}, Corollary 3.1]\label{89}
	Let $V$ be an absolute matrix order unit space and let $v_1, \dots, v_k \in M_{m,n}(V)$ be mutually orthogonal vectors for some $k,m,n \in \mathbb{N}$. Then $v_1, \dots, v_k \in \mathcal{PI}_{m,n}(V)$ if, and only if, $\displaystyle\sum_{i=1}^{k}v_i \in \mathcal{PI}_{m,n}(V).$ 
\end{proposition}

Next, we recall a result that we need in the section \ref{2}. To prove this result, we need to recall the following result:

\begin{proposition}[\cite{PI19}, Proposition 3.3]\label{4}
Let $V$ be an absolute matrix order unit space and let $\alpha \in M_{m,n}$ such that $\alpha^*\alpha=I_n.$ Then $\alpha u\alpha^* \perp \alpha v\alpha^*$ for $u,v\in M_n(V)^+$ with $u\perp v.$ In particular, if $p\in \mathcal{OP}_n(V),$ then $\alpha p\alpha^* \in \mathcal{OP}_m(V).$
\end{proposition}

\begin{corollary}\label{5}
Let $V$ be an absolue matrix order unit space. Also, let $v\in \mathcal{PI}_n(V)$ and $\alpha \in M_{m,n}$ such that $\alpha^*\alpha=I_n.$ Then $\alpha p\alpha^* \in \mathcal{PI}_m(V).$
\end{corollary}

\begin{proof}
Since $\alpha^* \alpha=I_n,$ we have $m\geq n$ and $\alpha=\delta \begin{bmatrix} I_n \\ 0_{m-n,n}\end{bmatrix}$ for some unitary matrix $\delta$. Then $\alpha v\alpha^*=\delta (v\oplus 0_{m-n})\delta^*$ and $(\alpha v\alpha^*)^*=\alpha v^*\alpha^*=\delta (v^*\oplus 0_{m-n})\delta^*$  and by Lemma \ref{95}, we get that $\vert \alpha v\alpha^* \vert_n=\delta (\vert v\vert_n \oplus 0_{m-n})\delta^*$ and $\vert (\alpha v\alpha^*) \vert_n=\delta (\vert v^*\vert_n \oplus 0_{m-n})\delta^*.$ Next, we have $\vert v\vert_n$ and $\vert v^*\vert_n \in \mathcal{OP}_n(V),$ thus by Proposition \ref{3}, we have $\vert v^*\vert_n \oplus 0_{m-n}$ and $\vert v^*\vert_n \oplus 0_{m-n} \in \mathcal{OP}_m(V).$ Finally, by Proposition \ref{4}, we get that $\alpha p\alpha^* \in \mathcal{PI}_m(V).$
\end{proof}

In \cite{KO21}, Karn and the author have described the direct limit of absolute matrix spaces. We denote direct limit of an absolute matrix order unit space $V$ by $M_\infty(V).$ The direct limit of an absolute matrix order unit space $V$ is identified with $\displaystyle \bigcup_{n=1}^\infty M_n(V)$ (see \cite[Theorems 3.13 and 3.14]{KO21}). Thus, we write: $M_\infty(V) = \displaystyle \bigcup_{n=1}^\infty M_n(V)$. Further, under this identification, the corresponding set of projections is identified with $\mathcal{OP}_\infty(V) = \displaystyle \bigcup_{n=1}^\infty \mathcal{OP}_n(V)$. For more information, we refer to \cite{KO21}.

In the next definition, a relation is defined on the set of all the order projections.

\begin{definition}[\cite{PI19}, Definition 4.1]
Let $V$ be an absolute matrix order unit space. Then we define a relation $\sim$ on $\mathcal{OP}_\infty(V)$ by the following way: 

Given $p\in \mathcal{OP}_m(V)$ and $q\in \mathcal{OP}_n(V)$, we say that $p$ is partial isometric equivalent to $q$, (we write, $p\sim q$), if there exists $v\in \mathcal{PI}_{m,n}(V)$ such that $p=\vert v^*\vert_{n,m}$ and $q=\vert v\vert_{m,n}$. 
\end{definition}
We recall the following properties of this relation.
\begin{proposition} [\cite{PI19}, Propositions 4.1 and 4.2] \label{15}
	Let $V$ be an absolute matrix order unit space and let $p,q,r,p',q'\in \mathcal{OP}_\infty(V)$. Then 
	\begin{enumerate}
		\item[(1)] If $m, n \in \mathbb{N}$ and let $p\in\mathcal{OP}_m(V)$, then $p\sim p\oplus 0_n$ and $p \sim 0_n \oplus p$; 
		\item[(2)] If $p \sim q$ and $p' \sim q'$ with $p \perp p'$ and $q \perp q'$, then $p + p' \sim q + q'$;
		\item[(3)] If $p\sim p'$ and $q\sim q',$ then $p\oplus q\sim p'\oplus q';$
		\item[(4)] $p\oplus q\sim q\oplus p;$
		\item[(5)] If $p,q\in \mathcal{OP}_n(V)$ for some $n\in \mathbb{N}$ such that $p\perp q,$ then $p+q\sim p\oplus q;$
		\item[(6)] $(p\oplus q)\oplus r=p\oplus (q\oplus r);$ 
		\item[(7)] $\sim$ is an equivalence relation on $\mathcal{OP}_\infty(V)$, if the following condition holds:
		\textbf{(T):} If $u \in \mathcal{PI}_{m,n}(V)$ and $v \in \mathcal{PI}_{l,n}(V)$ with $\vert u \vert_{m,n} = \vert v \vert_{l,n}$, then there exists a $w \in \mathcal{PI}_{m,l}(V)$ such that $\vert w^* \vert_{l,m} = \vert u^* \vert_{n,m}$ and $\vert w \vert_{m,l} = \vert v^* \vert_{n,l}$.
	\end{enumerate} 
\end{proposition} 

We extend $\sim$ to the following relation on $\mathcal{OP}_\infty(V)$:

\begin{definition}[\cite{KO21}, Definition 4.4]
Let $V$ be an absolute matrix order unit space. For $p,q\in \mathcal{OP}_\infty(V),$ we say that $p\approx q$, if there exists $r\in \mathcal{OP}_\infty(V)$ such that $p\oplus r\sim q\oplus r.$
\end{definition}

The newly defined relation $\approx$ enjoys the following properties:

\begin{proposition}[\cite{KO21}, Proposition 4.5]\label{16}
Let $V$ be an absolute matrix order unit space and let $p,q \in \mathcal{OP}_\infty(V)$. 
\begin{enumerate}
\item[(1)] $p \sim q$ implies $p \approx q$.  
\item[(2)] If $V$ satisfies condition (T), then $\approx$ is an equivalence relation. 
\end{enumerate}
\end{proposition}

Next, we recall one of the main result in \cite{KO21} around which whole theory of \cite{KO21} revolves.

\begin{theorem}[\cite{KO21}, Theorem 4.8]
Let $(V,e)$ be an absolute matrix order unit space satisfying (T) and consider $\mathcal{OP}_\infty(V) \times \mathcal{OP}_\infty(V)$. For all $p_1,p_2,q_1,q_2 \in \mathcal{OP}_\infty(V),$ we define $(p_1,q_1) \equiv (p_2,q_2)$ if and only if $p_1\oplus q_2 \approx p_2\oplus q_1.$ Then $\equiv$ is an equivalence relation on $\mathcal{OP}_\infty(V) \times \mathcal{OP}_\infty(V).$ Consider $$K_0(V)=\lbrace [(p_,q)]: p,q\in \mathcal{OP}_\infty\rbrace$$ where $[(p,q)]$ is the equivalence class of $(p,q)$ in $(\mathcal{OP}_\infty(V) \times \mathcal{OP}_\infty(V),\equiv).$ For all $p_1,p_2,q_1,q_2 \in \mathcal{OP}_\infty(V),$ we write $$[(p_1,q_1)]+ [(p_2,q_2)] = [(p_1\oplus p_2, q_1\oplus q_2)].$$ Then $(K_0(V),+)$ is an abelian group.
\end{theorem}

\begin{remark}[\cite{KO21}, Remark 4.9]
The abelian group $K_0(V)$ is called the $K_0(V)$-group of the absolute matrix order unit space $V.$ If $A$ is a $C^*$-algebra, then $K_0(A)$ is Grothendieck group of $A$ for projections. Thus $K_0$-group for absolute matrix order unit spaces is a generalization of $K_0$-group for $C^*$-algebras.
\end{remark}

For the definitions of group cone, proper cone, generating cone, ordered abelian group, ditinguished order unit and ordered abelian group with ditinguished order unit etc., see \cite[Subsection 4.1]{KO21} and for the definition of finite order projections, see \cite[Definition 5.1]{PI19}.  

The next two results tells that $K_0(V)$ receives order structure and order unit from an absolute order unit space $V$ in heredity so that it becomes an ordered abelian group with a distinguished order unit.

\begin{theorem}[\cite{KO21}, Theorem 4.10]
	Let $V$ be an absolute matrix order unit space satisfying (T). Put $K_0(V)^+ = \lbrace [(p,0)]: p\in \mathcal{OP}_\infty(V)\rbrace.$ Then 
	\begin{enumerate}
		\item[(a)]$K_0(V)^+$ is a group cone in $K_0(V)$. 
		\item[(b)]If $e^n$ is finite for all $n \in \mathbb{N}$, then $K_0(V)^+$ is proper. 
		\item[(c)] $K_0(V)^+$ is generating. 
	\end{enumerate} 
	In other words, if $e^n$ is finite for all $n \in \mathbb{N}$, then $(K_0(V),K_0(V)^+)$ is an ordered abelian group.
\end{theorem}

\begin{corollary}[\cite{KO21}, Corollary 4.11]
	Let $(V,e)$ be an absolute matrix order unit space satisfying (T) and let $e^n$ be finite for all $n \in \mathbb{N}$. Then $(K_0(V),K_0(V)^+)$ is an ordered abelian group with distinguished order unit $[(e,0)].$ In other words, for each $g\in K_0(V),$ there exists $n\in \mathbb{N}$ such that $-n[(e,0)]\le g\le n[(e,0)].$    
\end{corollary}

We recall suitable morphisms on absolute matrix order unit spaces, called completely $\vert \cdot \vert$-preserving maps. These maps are generalization of $\vert \cdot \vert$-preserving maps on absolute order unit spaces to absolute matrix order unit spaces \cite{K19}. The author with Karn has proved that under certain conditions such maps are isometries on these spaces (see \cite[Theorem 3.3 and 4.6]{K19}).

\begin{definition}[\cite{K19}, Definition 4.5]
Let $V$ and $W$ be absolute matrix order unit spaces and let $\phi:V\to W$ be a $\ast$-linear map. We say that $\phi$ is completely $\vert \cdot \vert$-preserving, if $\phi_n:M_n(V)\to M_n(W)$ is an $\vert \cdot \vert$-preserving map for each $n \in \mathbb{N}.$ Further, $\phi$ is said to be unital completely $\vert \cdot \vert$-preserving, if it is completely $\vert \cdot \vert$-preserving and it also preserves order unit.
\end{definition}

Let $V$ and $W$ be absolute matrix order unit spaces. We denote the zero map between $V$ and $W$ by $0_{W,V}.$ Similarly, the identity map on $V$ is denoted by $I_V.$ Furthermore, if $V$ and $W$ satisfy $(T),$ then we denote the zero group homomorphism between $K_0(V)$ and $K_0(W)$ by $0_{K_0(W),K_0(V)}$ and the identity map on $K_0(V)$ is denoted by $I_{K_0(V)}.$

Finally, the following result describes the functoriality of $K_0.$

\begin{theorem}[\cite{KO21}, Corollary 5.3 and Remark 5.4]
Let $U,V$ and $W$ be absolute matrix order unit spaces satisfying (T). Then 
\begin{enumerate}
\item[(a)]$K_0(I_V)=I_{K_0(V)};$
\item[(b)]If $\phi: U \to V$ and $\psi: V \to W$ be unital completely $\vert \cdot \vert$-preserving maps, then $K_0(\psi \circ \phi) = K_0(\psi) \circ K_0(\phi);$
\item[(c)]$K_0(0_{W,V}) = 0_{K_0(W),K_0(V)}.$
\end{enumerate}
In particular, $K_0$ is a functor from category of absolute matrix order unit spaces with morphisms as unital completely $\vert \cdot \vert$-preserving maps to category of abelian groups.
\end{theorem}

\section{$K_1$-group corresponding to an absolute matrix order unit space}

From this section, we begin our work. The notion of unitary elements in an absolute matrix order unit space has been introduced by Karn and the author in \cite{PI19}. In this section, we study basic properties of unitary elements in absolute matrix order unit spaces. We also study path homotopy equivalence of unitary elements in absolute matrix order unit spaces. By path homotopy equivalence, we define and study some variants of equivalence of unitary elements which will help us to describe the Grothendieck group $K_1(V)$ of an absolute matrix order unit space $V$ for unitary elements. The $K_1$-group for absolute matrix order unit spaces is a generalization of $K_1$-group for $C^*$-algebras. Later, we study order structure and functoriality of $K_1(V).$ Now, we begin with the following definition:

\begin{definition}[\cite{PI19}, Definition 3.1(2)]
Let $(V,e)$ be an absolute matrix order unit space and let $u \in M_n(V)$ for some $n\in \mathbb{N}.$ Then $u$ is said to be \emph{unitary}, if $\vert u \vert_n = e^n =\vert u^* \vert_n.$ We denote the set of all the unitary elements in $M_n(V)$ by $\mathcal{U}_n(V).$ Under the identification  $M_\infty(V) = \displaystyle \bigcup_{n=1}^\infty M_n(V),$ the corresponding set of unitaries is identified with $\mathcal{U}_\infty(V) = \displaystyle \bigcup_{n=1}^\infty \mathcal{U}_n(V).$
\end{definition}

We recall some properties of unitary elements.

\begin{proposition}\label{30}
Let $V$ be an absolute matrix order unit space. Then 
\begin{enumerate}
\item[(1)] If $m,n \in \mathbb{N}$ and $u\in \mathcal{U}_m(V),~v\in \mathcal{U}_n(V),$ then $u\oplus v=\begin{bmatrix} u & 0 \\ 0 & v \end{bmatrix}\in \mathcal{U}_{m+n}(V).$ In particular, $\oplus$ defines a binary operation on $\mathcal{U}_{\infty}(V).$
\item[(2)] If $v\in \mathcal{U}_n(V)$ and $\alpha \in \mathcal{U}_n(\mathbb{C}),$ then $\alpha^*v\alpha\in\mathcal{U}_n(V).$ 
\item[(3)] For each $v\in \mathcal{U}_n(V),w=\begin{bmatrix}0 & v \\ v^* & 0\end{bmatrix} \in \mathcal{U}_{2n}(V)$ with $w^*=w.$    
\end{enumerate}
\end{proposition}

\begin{proof}
Let $u \in \mathcal{U}_m(V),v\in \mathcal{U}_n(V)$ and $\alpha \in \mathcal{U}_n(\mathbb{C})$ for some $m,n\in \mathbb{N}.$ Then, by Definition \ref{152}(3), we get that $\vert u\oplus v\vert_{m+n}=\left \vert  \begin{bmatrix} u & 0 \\ 0 & v \end{bmatrix}\right \vert_{m+n} = \begin{bmatrix} \vert u\vert_m & 0 \\ 0 & \vert v\vert_n \end{bmatrix} = \begin{bmatrix} e^m & 0 \\ 0 & e^n\end{bmatrix}= e^{m+n}.$ Similarly, we get that $\vert (u\oplus v)^*\vert_{m+n}=e^{m+n}.$ Thus $u \oplus v \in \mathcal{U}_{m+n}(V).$

Next, by Lemma \ref{95}, we get that $|\alpha^*v\alpha|_n=\alpha^*|v|_n\alpha=\alpha^*e^n\alpha=\alpha^*\alpha e^n=I_ne^n=e^n.$ Similarly, we get that $|(\alpha^*v\alpha)^*|_n=\vert \alpha^*v^*\alpha\vert_n=e^n.$ Thus $\alpha^*v\alpha \in \mathcal{U}_n(V).$

Put $w=\begin{bmatrix}0 & v \\ v^* & 0\end{bmatrix}.$ Then $w^*=w$ and by Proposition \ref{91}(2), we have $\vert w\vert_{m+n}=\vert v^*\vert_n\oplus \vert v\vert_n=e^n\oplus e^n=e^{2n}$ so that $w \in \mathcal{U}_{2n}(V).$ 
\end{proof}

By Theorem \ref{1}, we know that absolute matrix order unit spaces are normed spaces. Thus, we can talk about norm continuous functions in an absolute matrix order unit space. Next, we study the path homotopy of unitary elements in absolute matrix order unit spaces.

\begin{definition}
Let $(V,e)$ be an absolute matrix order unit space. Let $u,v \in \mathcal{U}_n(V)$ for some $n \in \mathbb{N}.$ We say that $u$ is path homotopic to $v$ (we write it, $u\sim_h v$) if there exists a continuous function $f:[0,1] \to \mathcal{U}_n(V)$ such that $f(0)=u$ and $f(1)=v.$
\end{definition}

The path homotopy of unitaries is an equivalence realtion.

\begin{proposition}\label{33}
Let $(V,e)$ be an absolute matrix order unit space and let $n \in \mathbb{N}.$ Then $\sim_h$ is an equivalence relation on $\mathcal{U}_n(V).$
\end{proposition}

\begin{proof}
Let $u,v$ and $w \in \mathcal{U}_n(V).$ Put $f(t)=u$ for all $t \in [0,1].$ Then $f$ is a continuous function with $f(0)=u=f(1).$ Thus $u \sim_h u.$ 

If $f_1,f_2:[0,1]\to \mathcal{U}_n(V)$ be continuous functions with $f_1(0)=u, f_1(1)=v=f_2(0)$ and $f_2(1)=w,$ then the functions $g_1,g_2:[0,1]\to \mathcal{U}_n(V)$ defined by  $$g_1(t)=f_1(1-t)$$ and 

\[ 
   g_2(t) = \begin{cases}
    f_1(2t) & \text{at t $\in$ [0,$\frac{1}{2}$]}  \\
    f_2(2t-1) & \text{elsewhere}
   \end{cases}
\]

 are also continuous with $g_1(0)=v,g_1(1)=u=g_2(0)$ and $g_2(1)=w.$ Thus $u \sim_h v$ and $v \sim_h w$ implies $v \sim_h u$ and $u \sim_h w.$
\end{proof}

In next two results, we study some properties of the path homotopy of unitary elements in absolute matrix order unit spaces.

\begin{proposition}\label{32}
Let $V$ be an absolute matrix order unit space. Then 
\begin{enumerate}
\item[(1)] If $u\sim_h u'$ and $v\sim_h v'$ in $\mathcal{U}_m(V)$ and $\mathcal{U}_n(V)$ respectively, then $u\oplus v\sim_h u'\oplus v'$ in $\mathcal{U}_{m+n}(V).$    
\item[(2)] If $u\sim_h v$ and $\alpha \sim_h \beta$ in $~\mathcal{U}_n(V)$ and $~\mathcal{U}_n(\mathbb{C})$ respectively, then $\alpha^*u\alpha \sim_h \beta^*v\beta$ in $\mathcal{U}_n(V)$. 
\end{enumerate}
\end{proposition}

\begin{proof}
Let $f_1:[0,1]\to \mathcal{U}_m(V)$ and  $f_2:[0,1]\to \mathcal{U}_n(V)$ be continuous functions such that $f_1(0)=u,f_1(1)=u',f_2(0)=v$ and $f_2(1)=v'.$ Put $g(t)=f_1(t)\oplus f_2(t)$ for all $t\in [0,1].$ Then $g(0)=u\oplus v, g(1)=u'\oplus v'$ and by Proposition \ref{30}(1), we get that $g(t)\in \mathcal{U}_{m+n}(V)$ for all $t\in [0,1]$. By Theorem \ref{1}(2), we have $\|g(t_1)-g(t_2)\|_{m+n}=\max \lbrace \|f_1(t_1)-f_1(t_2)\|_m,\|f_2(t_1)-f_2(t_2)\|_n \rbrace$ for all $t_1,t_2\in [0,1].$ Since $f_1$ and $f_2$ are continuous, we get that $g:[0,1]\to \mathcal{U}_{m+n}(V)$ is also a continuous function. Thus $u\oplus v\sim_h u'\oplus v'$ in $\mathcal{U}_{m+n}(V).$  

Next, let $f_1:[0,1]\to \mathcal{U}_n(\mathbb{C})$ and  $f_2:[0,1]\to \mathcal{U}_n(V)$ are continuous functions such that $f_1(0)=\alpha,f_1(1)=\beta,f_2(0)=u$ and $f_2(1)=v.$ Define $f_1^*:[0,1]\to \mathcal{U}_n(\mathbb{C})$ such that $f_1^*(t)=f_1(t)^*.$ Since $\Vert \delta^*\Vert = \Vert \delta \Vert$ for all $\delta \in M_n$ and $f_1$ is continuous, thus $f_1^*$ is also continuous. Put $g(t)=f_1^*(t)f_2(t)f_1(t)$ for all $t\in [0,1].$ Then $g(0)=\alpha^*u\alpha,g(1)=\beta^*v\beta$ and by Proposition \ref{30}(2), we get that $g(t)\in \mathcal{U}_n(V)$ for all $t\in [0,1].$ Now $\|\alpha v\beta\|_n\leq\|\alpha\|\|v\|_n\|\beta\|$ for all $v \in M_n(V),\alpha$ and $\beta \in M_n$, therefore $g:[0,1]\to \mathcal{U}_n(V)$ is also   a continuous function. Hence $\alpha^*u\alpha \sim_h \beta^*v\beta$ in $\mathcal{U}_n(V)$.  
\end{proof}

\begin{lemma}\label{34}
Let $V$ be an absolute matrix order unit space. Then $u\oplus v\sim_h v\oplus u$ in $\mathcal{U}_{m+n}(V)$ for all $u\in \mathcal{U}_m(V)$ and $v\in \mathcal{U}_n(V).$
\end{lemma}

\begin{proof}
Let $u\in \mathcal{U}_m(V)$ and $v\in \mathcal{U}_n(V).$ Put $\alpha = \begin{bmatrix} 0_{m,n} & I_m \\ I_n & 0_{n,m} \end{bmatrix}.$ Then $\alpha \in \mathcal{U}_{m+n}(\mathbb{C})$ so that $\sigma(\alpha)\subsetneq \mathbb{T},$ where $\sigma(\alpha)$ and $\mathbb{T}$ denote the spectrum of $\alpha$ and the unit circle in $\mathbb{C}$ respectively. Next, by \cite[Lemma 2.1.3(ii)]{RLL00}, we have $I_{m+n}\sim_h z$ in $\mathcal{U}_{m+n}(\mathbb{C}).$ Thus, by Proposition \ref{32}(2), we get that $u\oplus v = I_{m+n}\cdot (u\oplus v)\cdot I_{m+n}\sim_h z^*\cdot (u\oplus v)\cdot z = v\oplus u$ in $\mathcal{U}_{m+n}(V).$  
\end{proof}

Now, we define a new relation on unitaries by help of path homotopy. 

\begin{definition}
Let $V$ be an absolute matrix order unit space. Define the relation $\sim_1$ on $\mathcal{U}_{\infty}(V)$ by given $u\in \mathcal{U}_m(V)$ and $v\in \mathcal{U}_n(V),~u\sim_1 v$ if and only if there exists $k\in \mathbb{N},~k>max \lbrace m,n\rbrace$ such that $u\oplus e^{k-m}\sim_h v\oplus e^{k-n}$ in $\mathcal{U}_k(V).$ 
\end{definition}

The following result describes that new relation is stronger one. 

\begin{corollary}\label{35}
Let $(V,e)$ be an absolute matrix order unit space. Then 
\begin{enumerate}
\item[(1)] $\sim_h$ implies $\sim_1$ in $\mathcal{U}_m(V)$ for any $m \in \mathbb{N}.$
\item[(2)] $\sim_1$ is an equivalence relation on $\mathcal{U}_\infty(V).$
\end{enumerate}
\end{corollary}

\begin{proof}
Let $u,v \in \mathcal{U}_m(V)$ with $u \sim_h v.$ Since $e^n \sim_h e^n$ for all $n \in \mathbb{N},$ by Proposition \ref{32}(1), we get $u \oplus e^{k-m} = u \oplus e^n \sim_h v\oplus e^n = v\oplus e^{k-m}$ for $k=m+n > m.$ Thus $u \sim_1 v$ in $\mathcal{U}_m(V).$

Next, let $u\in \mathcal{U}_l(V),v\in \mathcal{U}_m(V)$ and $w \in \mathcal{U}_n(V)$ with $u \sim_1 v$ and $v \sim_1 w.$ Then $u \sim_1 u$ as $u \sim_h u$ and there exist $k_1,k_2 \in \mathbb{N}$ with $k_1>\max \lbrace l,m\rbrace$ and  $k_2>\max \lbrace m,n\rbrace$ such that $u\oplus e^{k_1-l} \sim_h v\oplus e^{k_1-m}$ and $v\oplus e^{k_2-m} \sim_h w\oplus e^{k_2-n}$ in $\mathcal{U}_{k_1}(V)$ and $\mathcal{U}_{k_2}(V)$ respectively. Put $k=k_1+k_2$ so that $k>\max \lbrace l,m,n \rbrace$ and by Proposition \ref{32}(1), we have $$u\oplus e^{k-l}=(u\oplus e^{k_1-l})\oplus e^{k_2} \sim_h (v\oplus e^{k_1-m})\oplus e^{k_2}= v\oplus e^{k-m}$$ and $$v\oplus e^{k-m}=(v\oplus e^{k_2-m})\oplus e^{k_1} \sim_h (w\oplus e^{k_2-n})\oplus e^{k_1}=w\oplus e^{k-n}$$ in $\mathcal{U}_k(V).$ Now, by Proposition \ref{33}, we get $v\oplus e^{k-m} \sim_h u\oplus e^{k-l}$ and $u\oplus e^{k-l} \sim_h w\oplus e^{k-n}$ in $\mathcal{U}_k(V).$ Thus $v\sim_1 u$ and $u\sim_1 w$ in $\mathcal{U}_\infty(V).$

\end{proof}

We study some properties of newly defined equivalence relation.

\begin{proposition}\label{40}
Let $(V,e)$ be an absolute matrix order unit space and let $u,v,w,u',v'\in \mathcal{U}_{\infty}(V).$ Then
\begin{enumerate}
\item[(1)] $(u\oplus v)\oplus w=u\oplus (v\oplus w).$
\item[(2)] $u\sim_1 u\oplus e^n$ for all $n\in \mathbb{N}.$
\item[(3)] $u\oplus v\sim_1 v\oplus u.$
\item[(4)] If $u\sim_1 u'$ and $v\sim_1 v',$ then $u\oplus v\sim_1 u'\oplus v'.$
\end{enumerate}
\begin{proof}
Let $u \in \mathcal{U}_m(V).$ For fixed $n \in \mathbb{N},$ we can find $k \in \mathbb{N}$ such that $k > m+n.$ Then by Proposition \ref{33}, we get $u\oplus e^{k-m} \sim_h u\oplus e^{k-m}=(u\oplus e^n)\oplus e^{k-m-n}$ so that $u\sim_1 u\oplus e^n.$

By Lemma \ref{34}, we have $u\oplus v \sim_h v\oplus u.$ Now, by Corollary \ref{35}(1), we get $u\oplus v \sim_1 v\oplus u.$

Next, let $u' \in \mathcal{U}_l(V),v\in \mathcal{U}_{m'}(V)$ and $v'\in \mathcal{U}_{l'}(V)$ such that $u \sim_1 u'$ and $v\sim_1 v'.$ Then there exist $k_1,k_2 \in \mathbb{N},k_1 > \max \lbrace m,l\rbrace,k_2 > \max \lbrace m',l'\rbrace$ with $u\oplus e^{k_1-m}\sim_h u'\oplus e^{k1-l}$ and $v\oplus e^{k_2-m'}\sim_h v'\oplus e^{k_2-l'}.$ Put $k=k_1+k_2 > \lbrace m+m',l+l'\rbrace.$ Using part (1), Lemma \ref{34} and Propositions \ref{32}(1) and \ref{33} respectively, we get 
\begin{eqnarray*}
(u\oplus v)\oplus e^{k-m-m'} &\sim_h & (u\oplus e^{k_1-m})\oplus (v\oplus e^{k_2-m'})\\
&\sim_h & (u'\oplus e^{k_1-l})\oplus (v'\oplus e^{k_2-l'}) \\
&\sim_h & (u'\oplus v')\oplus e^{k-l-l'}
\end{eqnarray*}
so that $u\oplus v \sim_1 u'\oplus v'.$
\end{proof}
\end{proposition} 

Next, we again define a new relation on unitaries by $\sim_1.$

\begin{definition}
Let $V$ be an absolute matrix order unit space. For $u,v\in \mathcal{U}_\infty(V),$ we say that $u\approx_1 v,$ if there exists $w\in \mathcal{U}_\infty(V)$ such that $u\oplus w\sim_1 v\oplus w.$
\end{definition}

The following result shows that $\approx_1$ is stronger than $\sim_1.$

\begin{proposition}\label{61}
Let $V$ be an absolute matrix order unit space and let $n \in \mathbb{N}.$ Then 
\begin{enumerate}
\item[(1)] $\sim_1$ implies $\approx_1.$ 
\item[(2)] $\approx_1$ is an equivalence relation on $\mathcal{U}_\infty(V).$
\end{enumerate}
\end{proposition} 

\begin{proof}
By Proposition \ref{33}, we have $u\oplus w \sim_h u\oplus w$ for all $u,w\in \mathcal{U}_\infty(V).$ Then by Corollary \ref{35}(1), we get $u\oplus w \sim_1 u\oplus w$ so that $u\approx_1 u$ for all $u \in \mathcal{U}_\infty(V).$

Let $u_1,u_2,u_3 \in \mathcal{U}_\infty(V)$ such that  $u_1 \approx_1 u_2$ and $u_2 \approx_1 u_3.$ Then $u_1\oplus w_1 \sim_1 u_2 \oplus w_1$ and $u_2\oplus w_2\sim_1 u_3\oplus w_2$ for some $w_1,w_2 \in \mathcal{U}_\infty(V).$ By Corollary \ref{35}(2), we have $u_2\oplus w_1\sim_1 u_1\oplus w_1$ so that $u_2 \approx_1 u_1.$ By Corollary \ref{35} and by Proposition \ref{40}(4), we get $(u_1\oplus w_1)\oplus w_2 \sim_1 (u_2 \oplus w_1)\oplus w_2$ and $(u_2\oplus w_2)\oplus w_1\sim_1 (u_3\oplus w_2)\oplus w_1.$ Again applying Proposition \ref{40}(3), we have $w_1\oplus w_2 \sim_1 w_2 \oplus w_1.$ Then
\begin{eqnarray*}
u_1\oplus (w_1\oplus w_2) &=& (u_1\oplus w_1)\oplus w_2 \sim_1  (u_2\oplus w_1)\oplus w_2 \\
&= & u_2\oplus (w_1\oplus w_2)\sim_1  u_2\oplus (w_2\oplus w_1) \\
&= & (u_2\oplus w_2)\oplus w_1 \sim_1  (u_3\oplus w_2)\oplus w_1 \\
&= & u_3\oplus (w_2\oplus w_1)\sim_1 u_3\oplus (w_1\oplus w_2).
\end{eqnarray*} 
Finally, by Corollary \ref{35}(2), we conclude $u_1\oplus (w_1\oplus w_2)\sim_1  u_3\oplus (w_1\oplus w_2).$ Thus $u_1\approx_1 u_3.$
\end{proof}

The relation $\approx_1$ enjoys the following properties: 

\begin{corollary}\label{62}
Let $(V,e)$ be an absolute matrix order unit space and let $u,v,w,u',v'\in \mathcal{U}_\infty(V)$. Then 
\begin{enumerate}
\item[$1.$] $u\approx_1 u\oplus e^n$ for all $n\in \mathbb{N}.$
\item[$2.$] $u\oplus v\approx_1 v\oplus u.$
\item[$3.$] If $u\approx_1 u'$ and $v\approx_1 v',$ then $u\oplus v\approx_1 u'\oplus v'.$
\end{enumerate}
\end{corollary}

\begin{proof}
By Propositions \ref{40} and \ref{61}, it follows that $u\oplus v \approx_1 v\oplus u$ and $u\approx_1 u\oplus e^n$ for all $u,v \in \mathcal{U}_\infty(V)$ and $n \in \mathbb{N}.$

Next, let $u,u',v$ and $v' \in \mathcal{U}_\infty(V)$ such that  $u \approx_1 u'$ and $v \approx_1 v'.$ Then $u\oplus w_1 \sim_1 u' \oplus w_1$ and $v\oplus w_2\sim_1 v'\oplus w_2$ for some $w_1,w_2 \in \mathcal{U}_\infty(V).$ By Proposition \ref{40}, we get 
\begin{eqnarray*}
(u\oplus v)\oplus (w_1\oplus w_2) &=& u\oplus (v\oplus w_1)\oplus w_2 \sim_1  u\oplus (w_1\oplus v)\oplus w_2 \\
&= & (u\oplus w_1)\oplus (v\oplus w_2)\sim_1 (u'\oplus w_1)\oplus (v'\oplus w_2)  \\
&= & u'\oplus (w_1\oplus v')\oplus w_2 \sim_1  u'\oplus (v'\oplus w_1)\oplus w_2 \\
&= & (u'\oplus v')\oplus (w_1\oplus w_2).
\end{eqnarray*} 
Thus, by Corollary \ref{35}(2), we conclude that $(u\oplus v)\oplus (w_1\oplus w_2)\sim_1 (u'\oplus v')\oplus (w_1\oplus w_2)$ so that $u\oplus v\approx_1 u'\oplus v'.$
\end{proof}

Now, we define a binary operatoin in the family of equivalence classes of unitaries under $\approx_1.$ 

\begin{proposition}\label{63}
Let $(V,e)$ be an absolute matrix order unit space. For each $u,v\in \mathcal{U}_\infty(V),$ let $[u] = \lbrace w\in \mathcal{U}_\infty(V): w\approx_1 u \rbrace$ and put $[u]+ [v] = [u\oplus v].$ Then, $+$ is a binary operation in the family of equivalence classes $(\mathcal{U}_\infty(V)/\approx_1,+)$. Also:
\begin{enumerate}
\item[(1)] $[u] + [e] = [u]$ for all $u\in \mathcal{U}_\infty(V);$
\item[(2)] $[u]+ [v] = [v]+ [u]$ for all $u,v\in \mathcal{U}_\infty(V);$
\item[(3)] $[u]+ [w] = [v]+ [w]$ for $u,v,w\in \mathcal{U}_\infty(V),$ then $[u] = [v].$ 
\end{enumerate}
Thus, $(\mathcal{U}_\infty(V)\big/\approx_1, +)$ is an abelian semi-group satisfying the cancellation law.
\end{proposition}
\begin{proof}
By Corollary \ref{62}(3), it follows that $+$ is well-defined on $\mathcal{U}_\infty(V)\big/\approx_1.$ Note that (1) and (2) immediately follow from \ref{62}(1) and (2) respectively. Next, we prove (3).

Let $u,v,w\in \mathcal{U}_\infty(V)$ such that $[u]+ [w] = [v]+ [w].$ Then $u\oplus w \approx_1 v\oplus w$ so that $u\oplus (w \oplus x) \sim_1 v\oplus (w \oplus x)$ for some $x\in \mathcal{U}_\infty(V).$ Thus, $u\approx_1 v$ so that $[u] = [v].$
\end{proof}

Finally, the following result shows the importance of path homotopy of unitaries to define a group. 

\begin{theorem}\label{64}
Let $(V,e)$ be an absolute matrix order unit space and consider $\mathcal{U}_\infty(V) \times \mathcal{U}_\infty(V).$ For all $u_1,u_2,v_1,v_2 \in \mathcal{U}_\infty(V),$ we define $(u_1,v_1) \equiv_1 (u_2,v_2)$ if and only if $u_1\oplus v_2 \approx_1 u_2\oplus v_1.$ Then, $\equiv_1$ is an equivalence relation on $\mathcal{U}_\infty(V) \times \mathcal{U}_\infty(V).$ Consider: $$K_1(V)=\lbrace [(u,v)]: u,v\in \mathcal{U}_\infty(V)\rbrace,$$ where $[(u,v)]$ is the equivalence class of $(u,v)$ in $(\mathcal{U}_\infty(V) \times \mathcal{U}_\infty(V),\equiv_1).$ For all $u_1,u_2,v_1,v_2 \in \mathcal{U}_\infty(V),$ we write: $$[(u_1,v_1)]+ [(u_2,v_2)] = [(u_1\oplus u_2, v_1\oplus v_2)].$$ Then $(K_1(V),+)$ is an abelian group.
\end{theorem}

\begin{proof}
It follows from Proposition \ref{61}(2) and Corollary \ref{62} that the relation $\equiv_1$ on $\mathcal{U}_\infty(V) \times \mathcal{U}_\infty(V)$ is reflexive and symmetric. Let $(u_1,v_1) \equiv_1 (u_2,v_2)$ and $(u_2,v_2) \equiv_1 (u_3,v_3)$ for some $u_1,u_2,u_3,v_1,v_2,v_3 \in \mathcal{U}_\infty(V).$ Then $u_1\oplus v_2 \approx_1 u_2\oplus v_1$ and $u_2\oplus v_3 \approx_1 u_3\oplus v_2.$ By Proposition \ref{61}(2) and by Corollary \ref{62}(2) and (3), we get that $(u_1\oplus v_3)\oplus (u_2 \oplus v_2)\approx_1 (u_3\oplus v_1)\oplus (u_2\oplus v_2)$ so that $[u_1\oplus v_3]+ [u_2 \oplus v_2] = [u_3\oplus v_1]+[u_2\oplus v_2].$ Then, by Proposition \ref{63}(3), we conclude that $[u_1\oplus v_3] = [u_3\oplus v_1]$ so that $u_1\oplus v_3 \approx_1 u_3\oplus v_1.$ Thus $(u_1,v_1) \equiv_1 (u_3,v_3)$ so that $\equiv_1$ is transitive. Hence, $\equiv_1$ is an equivalence relation on $\mathcal{U}_\infty(V) \times \mathcal{U}_\infty(V).$

Now, we show that $+$ is well defined on $K_1(V).$ Let $[(u_1,v_1)]=[(u_1',v_1')]$ and $[(u_2,v_2)]=[(u_2',v_2')].$ Then $(u_1,v_1)\equiv_1(u_1',v_1')$ and $(u_2,v_2)\equiv_1(u_2',v_2')$ so that $u_1\oplus v_1'\approx_1 u_1'\oplus v_1$ and  $u_2\oplus v_2'\approx_1 u_2'\oplus v_2.$ Again by Proposition \ref{61}(2) and by Corollary \ref{62}(2) and (3), we get that $(u_1\oplus u_2)\oplus (v_1'\oplus v_2')\approx_1 (u_1'\oplus u_2')\oplus (v_1\oplus v_2).$ Thus $[(u_1,v_1)]+ [(u_2,v_2)]=[(u_1',v_1')]+ [(u_2',v_2')]$ so that $+$ is well defined.

By Corollary \ref{62}(2) and (3), we have: $$(u_1\oplus u_2)\oplus (v_2\oplus v_1) \approx_1 (u_2\oplus u_1)\oplus (v_1\oplus v_2)$$ for all $u_1,u_2,v_1,v_2\in \mathcal{U}_\infty(V),$ so that $+$ is commutative in $K_1(V).$

Let $u,v\in \mathcal{U}_\infty(V).$ By Corollary \ref{62}(1) and (3), we have $(u\oplus e)\oplus (v\oplus e) \approx_1 u\oplus v.$ Thus $[(e,e)]$ is an identity element in $K_1(V).$  

Associativity of $+$ on $K_1(V)$ follows from Proposition \ref{40}(1).

For any $u,v \in \mathcal{U}_\infty(V),$ we have $[(u\oplus v,v\oplus u)]=[(e,e)].$ Hence $[(v,u)]$ is an inverse element of $[(u,v)]$ in $K_1(V).$
\end{proof}

\begin{remark}
The abelian group $K_1(V)$ is called the $K_1(V)$-group of the absolute matrix order unit space $V.$ If $A$ is a $C^*$-algebra, then $K_1(A)$ is Grothendieck group of $A$ for unitary elements. Thus $K_1$-group for absolute matrix order unit spaces is a generalization of $K_1$-group for $C^*$-algebras.
\end{remark}

\subsection{Order structure in $K_1$}

In this subsection, we prove that $K_1(V)$ bears order structure on it. We start with the following discussion:

Let $A$ be a unital $C^*$-algebra with unity $1_A.$ Put $K_1(A)^+ = \lbrace [(v,1_A)]: v\in \mathcal{U}_\infty(A)\rbrace.$ Then
$K_1(A)^+$ is a group cone and generating. But $K_1(A)^+$ is not proper. By Whitehead lemma $($see \cite[Lemma 2.1.5]{RLL00}$)$, we have $v\oplus v^* \sim_h 1_A^{2n}$ for any $v\in \mathcal{U}_n(V),n\in \mathbb{N}.$ Then $[(v,1_A)]=[(v,1_A^n)]=[(1_A^n,v^*)]=[(1_A,v^*)]$ so that $\pm [(v,1_A)] \in K_1(A)^+.$ Thus $K_1(A)^+$ is not proper.

Under certain conditions, we prove that $(K_1(V),K_1(V)^+)$ is an ordered abelian group with a distinguished order unit. For that, we need to prove the following result:

\begin{theorem}\label{94}
	Let $V$ be an absolute matrix order unit space. Put $K_1(V)^+ = \lbrace [(v,e)]: v\in \mathcal{U}_\infty(V)\rbrace.$ Then 
	\begin{enumerate}
		\item[(1)]$K_1(V)^+$ is a group cone in $K_1(V)$.
		\item[(2)]$K_1(V)^+$ is generating. 
		\item[(3)] If $v\oplus v^* \sim_h e^{2n}$ in $\mathcal{U}_{2n}(V)$ for all $n \in \mathbb{N}$ and $v\in \mathcal{U}_n(V),$ then for each $g\in K_1(V),$ there exists $m\in \mathbb{N}$ such that $-m[(e,e)]\leq g\leq m[(e,e)].$    
	\end{enumerate} 
\end{theorem}

\begin{proof}
	It is routine to verify (1) and (2). Next, we prove (3). Assume that $v\oplus v^* \sim_h e^{2n}$ in $\mathcal{U}_{2n}(V)$ for all $n \in \mathbb{N}$ and $v\in \mathcal{U}_n(V).$ First, let $v\in \mathcal{U}_n(V)$ for some $n \in \mathbb{N}.$ By part (1), we have 
	\begin{eqnarray*}
		[(v,e)] &=& [(v,e^n)] \\
		& \leq & [(v,e^n)]+[(v^*,e^n)] \\
		&=& [(v \oplus v^*,e^n\oplus e^n)] \\
		&=& [(e^{2n},e^{2n})]\\
		&=& 2n[(e,e)].
	\end{eqnarray*}
	
	Next, let $g\in  K_1(V).$ Then by part (2), we have $g=[(u,e)]-[(v,e)]$ for some $u,v\in \mathcal{U}_\infty(V).$ Without loss of generality, we may assume that $u,v\in \mathcal{U}_n(V)$ for some $n\in \mathbb{N}.$ Since $-[(v,e)] \leq g \leq [(u,e)],$ we get that $-2n[(e,e)] \leq g \leq 2n[(e,e)].$ Put $m=2n.$ Thus $-m[(e,e)]\leq g\leq m[(e,e)].$	
\end{proof}

\begin{corollary}
Let $(V,e)$ be an absolute matrix order unit space. Then by Theorem \ref{94}, we conclude that $(K_1(V),K_1(V)^+)$ is an ordered abelian group. Moreover, if $v\oplus v^* \sim_h e^{2n}$ in $\mathcal{U}_{2n}(V)$ for all $n \in \mathbb{N}$ and $v\in \mathcal{U}_n(V),$ then $(K_1(V),K_1(V)^+)$ is an ordered abelian group with distinguished order unit $[(e,e)].$ 
\end{corollary}

\subsection{Functoriality of $K_1$}

Let $V$ be an absolute matrix order unit space. Then $v \longmapsto [(v,e)]$ defines a map $\Upsilon_{V}:\mathcal{U}_\infty(V) \to K_1(V)$. If $V$ and $W$ are complex vector spaces and if $\phi: V \to W$ be a linear map, we denote the corresponding map from $M_\infty(V)$ to $M_\infty(W)$ again by $\phi$. In this sense, $\phi_{\mid_{M_n(V)}} = \phi_n$ for all $n \in \mathbb{N}$. In the next result, we describe the functorial nature of $K_1$. For this, first we prove the following commutative property of $K_1(V)$:

\begin{theorem}\label{79}
Let $(V,e_V)$ and $(W,e_W)$ be absolute matrix order unit spaces and let $\phi: V \to W$ be a unital completely $\vert \cdot \vert$-preserving map. Then there exists a unique group homomorphism $K_1(\phi): K_1(V)\to K_1(W)$ such that the following diagram commutes: 
\end{theorem}

$$\begin{tikzcd}
\mathcal{U}_\infty(V)\arrow{r}{\phi}\arrow[swap]{d}{\Upsilon_V}
& \mathcal{U}_\infty(W)\arrow{d}{\Upsilon_W} \\
K_1(V)\arrow[swap]{r}{K_1(\phi)}
& K_1(W)
\end{tikzcd}$$

\begin{proof}
Let $n\in \mathbb{N}$ and $v\in \mathcal{U}_n(V).$ Since $\vert v\vert_n=e^n=\vert v^*\vert_n$ and $\phi$ is unital completely $\vert \cdot\vert$-preserving map, we get that $\vert \phi(v)\vert_n=\phi(\vert v\vert_n)=\phi(e_V^n)=\phi(e_V)^n=e_W^n$ and $\vert \phi(v)^*\vert_n=\vert\phi(v^*)\vert_n=\phi(\vert v^*\vert_n)=\phi(e_v^n)=e_W^n$ so that $\phi(v)\in \mathcal{U}_n(W).$ Thus $\phi(\mathcal{U}_\infty(V))\subset \mathcal{U}_\infty(W).$ 

By \cite[Theorems 3.7 and 3.8]{K19}, we get that $\phi$ is a contraction on ${M_n(V)}_{sa}$ for each $n\in \mathbb{N}.$ By Theorem \ref{1}(5), we conclude that

\begin{eqnarray*}
\Vert \phi(v)\Vert_n &=& \left \Vert \begin{bmatrix} 0 & \phi(v)\\ \phi(v)^* & 0\end{bmatrix} \right \Vert_{2n} \\
&=& \left \Vert \begin{bmatrix} 0 & \phi(v)\\ \phi(v^*) & 0\end{bmatrix} \right \Vert_{2n} \\
&=& \left \Vert \phi \left(\begin{bmatrix} 0 & v \\ v^* & 0\end{bmatrix}\right) \right \Vert_{2n} \\
&\leq & \left \Vert \begin{bmatrix} 0 & v \\ v^* & 0\end{bmatrix} \right \Vert_{2n} \\
&=& \Vert v\Vert_n
\end{eqnarray*} 

so that $\phi_n$ is a contraction on $M_n(V)$. Then $\phi$ is completely contraction.

Now, let $w \in \mathcal{U}_n(V)$ such that $v \sim_h w$. As $\phi_n$ is continuous, we get $\phi(v) \sim_h \phi(w)$ in $\mathcal{U}_n(W)$.

Next, put $K_1(\phi)([(v,w)])=[(\phi(v),\phi(w))]$ for each $[(v,w)]\in K_1(V)$. We show that $K_1(\phi)$ is well-defined. Let $[(v_1,w_1)] = [(v_2,w_2)]$ for some $v_1,v_2,w_1,w_2 \in \mathcal{U}_\infty(V).$ Then there exists $u\in \mathcal{U}_\infty(V)$ such that $v_1\oplus w_2 \oplus u \sim_h v_2 \oplus w_1 \oplus u.$ Thus $\phi(v_1)\oplus \phi(w_2) \oplus \phi(u) \sim_h \phi(v_2) \oplus \phi(w_1) \oplus \phi(u)$ so that $[(\phi(v_1),\phi(w_1))] = [(\phi(v_2),\phi(w_2))].$ Hence $K_1(\phi)$ is well-defined. For all $[(v_1,w_1)],[(v_2,w_2)] \in K_1(V)$, we have that 

\begin{eqnarray*}
K_1(\phi)([(v_1,w_1)]+[(v_2,w_2)]) &=& K_1(\phi)([(v_1 \oplus v_2,w_1 \oplus w_2)]) \\
&=& [(\phi(v_1 \oplus v_2),\phi(w_1 \oplus w_2))] \\
&=& [(\phi(v_1) \oplus \phi(v_2),\phi(w_1) \oplus \phi(w_2))] \\
&=& [(\phi(v_1),\phi(w_1))] + [(\phi(v_2),\phi(w_2))] \\
&=& K_1(\phi)([(v_1,w_1)]) + K_1(\phi)([(v_2,w_2)])
\end{eqnarray*}

so that $K_1(\phi)$ is a group homomorphism. By construction $K_1$ satisfies the diagram.

\textbf{Uniqueness of $K_1(\phi)$:-} Let $\mathcal{H}: K_1(V) \to K_1(W)$ be a group homomorphism satisfying the same diagram. Then $K_1(\phi)(\Upsilon_V(v))=\Upsilon_W(\phi(v))=\mathcal{H}(\Upsilon_V(v))$ for all $v \in \mathcal{U}_\infty(V)$. Thus we get that

\begin{eqnarray*}
K_1(\phi)([(v,w)]) &=& K_1(\phi)([(v,e)]-[(w,e)]) \\
&=& K_1(\phi)(\Upsilon_V(v)-\Upsilon_V(w)) \\
&=& K_1(\phi)(\Upsilon_V(v)) - K_1(\phi)(\Upsilon_V(w)) \\
&=& \mathcal{H}(\Upsilon_V(v)) - \mathcal{H}(\Upsilon_V(w)) \\
&=& \mathcal{H}(\Upsilon_V(v)-\Upsilon_V(w)) \\
&=& \mathcal{H}([(v,w)])
\end{eqnarray*} 

for all $[(v,w)] \in K_1(V)$. Hence $K_1(\phi) = \mathcal{H}$.
\end{proof}

Let $V$ and $W$ be absolute matrix order unit spaces. We denote the identity maps on $K_1(V)$ by $I_{K_1(V)}.$  

\begin{corollary}\label{80}
Let $U,V$ and $W$ be absolute matrix order unit spaces, and $\phi: U \to V$ and $\psi: V \to W$ be unital completely $\vert \cdot \vert$-preserving maps. Then 
\begin{enumerate}
\item[(a)]$K_1(I_V)=I_{K_1(V)};$
\item[(b)]$K_1(\psi \circ \phi) = K_1(\psi) \circ K_1(\phi),$
\end{enumerate}
\end{corollary}

\begin{proof}
\begin{enumerate}
\item[(a)] Let $v,w \in \mathcal{U}_\infty(V)$. Then
\begin{eqnarray*}
K_1(I_V)([(v,w)]) &=& [(I_V(v),I_V(w))] \\
&=& [(v,w)] 
\end{eqnarray*}

so that by Theorem \ref{79}, $K_1(I_V)=I_{K_1(V)}$.
\item[(b)] For any $[(u,v)] \in K_1(U)$, we get that
\begin{eqnarray*}
K_1(\psi \circ \phi)([(u,v)]) &=& [(\psi \circ \phi(u),\psi \circ \phi(v))] \\
&=& [(\psi(\phi(u)),\psi(\phi(v)))] \\
&=& K_1(\psi)[(\phi(u),\phi(v))] \\
&=& K_1(\psi)(K_1(\phi)([(u,v)])) \\
&=& K_1(\psi) \circ K_1(\phi)([(u,v)]). 
\end{eqnarray*}
Thus, again by Theorem \ref{79}, we conclude that $K_1(\psi \circ \phi) = K_1(\psi) \circ K_1(\phi)$.
\end{enumerate}
\end{proof}

\begin{remark}
It follows from Corollary \ref{80} that $K_1$ is a functor from category of absolute matrix order unit spaces with morphisms as unital completely $\vert \cdot \vert$-preserving maps to category of abelian groups.
\end{remark}

In the end of the subsection, we note the following isomorphic property of $K_1.$ 

\begin{corollary}
Let $V$ and $W$ be isomorphic absolute matrix order unit spaces (isomorphic in the sense that there exists a unital, bijective completely $\vert \cdot \vert$-preserving map between $V$ and $W$). Then $K_1(V)$ and $K_1(W)$ are group isomorphic. 
\end{corollary}

\begin{proof}
Let $\phi:V \to W$ be unital completely $\vert \cdot \vert$-preserving map. Then $\phi^{-1}$ is also unital completely $\vert \cdot \vert$-preserving map. Since $\phi^{-1} \circ \phi = I_{V}$ and $\phi \circ \phi^{-1} = I_{W}$, by Corollary \ref{80}(a) and (b), we get that $K_1(\phi^{-1}) \circ K_1(\phi) = I_{K_1(V)}$ and $K_1(\phi) \circ K_1(\phi^{-1}) = I_{K_1(W)}$. Thus $K_1(\phi):K_1(V) \to K_1(W)$ is a surjective group isomorphism and $K_1(\phi)^{-1} = K_1(\phi^{-1})$. Hence $K_1(V)$ and $K_1(W)$ are group isomorphic. 
\end{proof}

\section{$K$-group corresponding to an absolute matrix order unit space}\label{2}

This section is very similar to the previous one except few results. Nevertheless, for the sake of completeness and accuracy of results, let's have a quick look at all the results with proofs as we need some results of this section for the next section, where we derive a relation among $K_0(V),K_1(V)$ and $K(V)$ for an absolute matrix order unit space $V$. Note that $K(V)$ is described in this section. Let's start the section.

The notion of partial unitary elements in an absolute matrix order unit space has been introduced by Karn and the author in \cite{PI19}. In this section, we study basic properties of partial unitary elements in absolute matrix order unit spaces. We also study path homotopy equivalence of partial unitary elements in absolute matrix order unit spaces. By path homotopy equivalence, we define and study some variants of equivalence of partial unitary elements which will help us to describe the Grothendieck group $K(V)$ of an absolute matrix order unit space $V$ for partial unitary elements. The $K$-group for absolute matrix order unit spaces is a generalization of $K$-group for $C^*$-algebras. Later, we study order structure and functoriality of $K(V).$ Now, we begin with the following definition:

\begin{definition}[\cite{PI19}, Definition 3.1(5)]\label{86}
Let $V$ be an absolute matrix order unit space and let $u \in M_n(V)$ for some $n\in \mathbb{N}.$ Then $u$ is said to be \emph{partial unitary}, if $\vert u \vert_n = \vert u^* \vert_n$ is an order projection. We denote the set of all the partial unitary elements in $M_n(V)$ by $\mathcal{PU}_n(V).$ Under the identification  $M_\infty(V) = \displaystyle \bigcup_{n=1}^\infty M_n(V),$ the corresponding set of partial unitaries is identified with $\mathcal{PU}_\infty(V) = \displaystyle \bigcup_{n=1}^\infty \mathcal{PU}_n(V).$

\end{definition}

We recall some properties of partial unitary elements.

\begin{proposition}\label{65}
Let $V$ be an absolute matrix order unit space. Then 
\begin{enumerate}
\item[(1)] If $m,n \in \mathbb{N}$ and $u\in \mathcal{PU}_m(V),~v\in \mathcal{PU}_n(V),$ then $u\oplus v=\begin{bmatrix} u & 0 \\ 0 & v \end{bmatrix}\in \mathcal{PU}_{m+n}(V).$ In particular, $\oplus$ defines a binary operation on $\mathcal{PU}_\infty(V).$
\item[(2)] If $v\in \mathcal{PU}_n(V)$ and $\alpha \in \mathcal{U}_n(\mathbb{C}),$ then $\alpha^*v\alpha \in \mathcal{PU}_n(V).$    
\item[(3)] For each $v\in \mathcal{PI}_{m,n}(V),w=\begin{bmatrix}0 & v \\ v^* & 0\end{bmatrix} \in \mathcal{PI}_{m+n}(V)$ with $w^*=w$ so that $w\in \mathcal{PU}_{m+n}(V)$    
\end{enumerate}
\end{proposition}

\begin{proof}
Let $u\in \mathcal{PU}_m(V)$ and $v\in \mathcal{PU}_n(V).$ Then $\vert u\vert_m=\vert u^*\vert_m \in \mathcal{OP}_m(V)$ and $\vert v\vert_n= \vert v^*\vert_n \in \mathcal{OP}_n(V).$ By Proposition \ref{4}, we get that $\vert u\oplus v\vert_{m+n} = \vert u\vert_m \oplus \vert v\vert_n = \vert u^*\vert_m\oplus \vert v^*\vert_n = \vert u^*\oplus v^*\vert_{m+n} =\vert (u\oplus v)^*\vert_{m+n}\in \mathcal{OP}_{m+n}(V).$ Thus $u \oplus v \in \mathcal{PU}_{m+n}(V).$

Next, by Lemma \ref{95} and by Corollary \ref{5}, we get $|\alpha^*v\alpha|_n=\alpha^*|v|_n\alpha=\alpha^* \vert v^* \vert_n \alpha=\vert (\alpha^* v\alpha)^*\vert_n $ and $\alpha^*v\alpha \in \mathcal{PI}_n(V)$. Thus $\alpha^*v\alpha\in \mathcal{PU}_n(V).$

Let $v\in \mathcal{PI}_{m,n}(V).$ Thus $\vert v^*\vert_{n,m}\in \mathcal{OP}_m(V)$ and $\vert v\vert_{m,n}\in \mathcal{OP}_n(V).$ Put $w=\begin{bmatrix}0 & v \\ v^* & 0\end{bmatrix}.$ Then $w^*=w$ and by Proposition \ref{4}, we have $\vert w\vert_{2n}=\vert v^*\vert_n\oplus \vert v\vert_n \in \mathcal{OP}_{m+n}(V)$ so that $w \in \mathcal{PI}_{m+n}(V).$ Hence $w\in \mathcal{PU}_{m+n}(V).$
\end{proof}

Next, we study the path homotopy of partial unitary elements in absolute matrix order unit spaces.

\begin{definition}\label{66}
Let $V$ be an absolute matrix order unit space. Let $u,v \in \mathcal{PU}_n(V)$ for some $n \in \mathbb{N}.$ We say that $u$ is path homotopic to $v$ (we continue to write it, $u\sim_h v$) if there exists a continuous function $f:[0,1] \to \mathcal{PU}_n(V)$ such that $f(0)=u$ and $f(1)=v.$
\end{definition}

Following the proof mentioned in Proposition \ref{33}, we see that the path homotopy of partial unitaries is also an equivalence realtion.

\begin{proposition}\label{67}
Let $V$ be an absolute matrix order unit space and let $n \in \mathbb{N}.$ Then $\sim_h$ is an equivalence relation on $\mathcal{PU}_n(V).$
\end{proposition}

In next two results, we study some properties of the path homotopy of partial unitary elements in absolute matrix order unit spaces.

\begin{proposition}\label{68}
Let $V$ be an absolute matrix order unit space. Then 
\begin{enumerate}
\item[(1)] If $u\sim_h u'$ and $v\sim_h v'$ in $\mathcal{PU}_m(V)$ and $\mathcal{PU}_n(V)$ respectively, then $u\oplus v\sim_h u'\oplus v'$ in $\mathcal{PU}_{m+n}(V).$    
\item[(2)] If $u\sim_h v$ and $\alpha \sim_h \beta$ in $~\mathcal{PU}_n(V)$ and $~\mathcal{U}_n(\mathbb{C})$ respectively, then $\alpha^*u\alpha \sim_h \beta^*v\beta$ in $\mathcal{PU}_n(V)$. 
\end{enumerate}
\end{proposition}

\begin{proof}
Let $f_1:[0,1]\to \mathcal{PU}_m(V)$ and  $f_2:[0,1]\to \mathcal{PU}_n(V)$ be continuous functions such that $f_1(0)=u,f_1(1)=u',f_2(0)=v$ and $f_2(1)=v'.$ Define $g:[0,1]\to \mathcal{PU}_{m+n}(V)$ such that $g(t)=f_1(t)\oplus f_2(t)$ for all $t\in [0,1].$ Then $g(0)=u\oplus v, g(1)=u'\oplus v'$ and $g$ is also continuous. Thus $u\oplus v\sim_h u'\oplus v'$ in $\mathcal{PU}_{m+n}(V).$  

Next, let $f_1:[0,1]\to \mathcal{U}_n(\mathbb{C})$ and  $f_2:[0,1]\to \mathcal{PU}_n(V)$ are continuous functions such that $f_1(0)=\alpha,f_1(1)=\beta,f_2(0)=u$ and $f_2(1)=v.$ Define $g:[0,1]\to \mathcal{PU}_n(V)$ such that $g(t)=f_1^*(t)f_2(t)f_1(t)$ for all $t\in [0,1].$ Then $g(0)=\alpha^*u\alpha,g(1)=\beta^*v\beta$ and $g$ is continuous. Hence $\alpha^*u\alpha \sim_h \beta^*v\beta$ in $\mathcal{PU}_n(V).$
\end{proof}

\begin{lemma}\label{69}
Let $V$ be an absolute matrix order unit space. Then $u\oplus v\sim_h v\oplus u$ in $\mathcal{PU}_{m+n}(V)$ for all $u\in \mathcal{PU}_m(V)$ and $v\in \mathcal{PU}_n(V).$
\end{lemma}

\begin{proof}
Again, by $I_{m+n}\sim_h z$ in $\mathcal{U}_{m+n}(\mathbb{C}),$ we get that $u\oplus v = I_{m+n}\cdot (u\oplus v)\cdot I_{m+n}\sim_h z^*\cdot (u\oplus v)\cdot z = v\oplus u$ in $\mathcal{PU}_{m+n}(V).$ 
\end{proof}

Now, we define a new relation on partial unitaries by help of path homotopy.

\begin{definition}
Let $V$ be an absolute matrix order unit space. Define the relation $\sim_K$ on $\mathcal{PU}_{\infty}(V)$ by given $u\in \mathcal{PU}_m(V)$ and $v\in \mathcal{PU}_n(V),~u\sim_K v$ if and only if there exists $k\in \mathbb{N},~k>max \lbrace m,n\rbrace$ such that $u\oplus 0_{k-m}\sim_h v\oplus 0_{k-n}$ in $\mathcal{PU}_k(V).$ 
\end{definition}

The following result describes that new relation is stronger one.

\begin{corollary}\label{70}
Let $V$ be an absolute matrix order unit space. Then 
\begin{enumerate}
\item[(1)] $\sim_h$ implies $\sim_K$ in $\mathcal{PU}_m(V)$ for any $m \in \mathbb{N}.$
\item[(2)] $\sim_K$ is an equivalence relation on $\mathcal{PU}_\infty(V).$
\end{enumerate}
\end{corollary}

\begin{proof}
Let $m\in \mathbb{N}$ and $u,v \in \mathcal{PU}_m(V)$ with $u \sim_h v.$ Since $0_n \sim_h 0_n$ for all $n \in \mathbb{N},$ by Proposition \ref{68}(1), we get $u \oplus 0_{k-m} = u \oplus 0_n \sim_h v\oplus 0_n = v\oplus 0_{k-m}$ for $k=m+n > m.$ Thus $u \sim_K v$ in $\mathcal{PU}_m(V).$

Next, let $u\in \mathcal{PU}_l(V),v\in \mathcal{PU}_m(V)$ and $w \in \mathcal{PU}_n(V)$ with $u \sim_K v$ and $v \sim_K w.$ Then $u \sim_K u$ as $u \sim_h u$ and there exist $k_1,k_2 \in \mathbb{N}$ with $k_1>\max \lbrace l,m\rbrace$ and  $k_2>\max \lbrace m,n\rbrace$ such that $u\oplus 0_{k_1-l} \sim_h v\oplus 0_{k_1-m}$ and $v\oplus 0_{k_2-m} \sim_h w\oplus 0_{k_2-n}$ in $\mathcal{PU}_{k_1}(V)$ and $\mathcal{PU}_{k_2}(V)$ respectively. Put $k=k_1+k_2$ so that $k>\max \lbrace l,m,n \rbrace$ and by Proposition \ref{68}(1), we have $$u\oplus 0_{k-l}=(u\oplus 0_{k_1-l})\oplus 0_{k_2} \sim_h (v\oplus 0_{k_1-m})\oplus 0_{k_2}= v\oplus 0_{k-m}$$ and $$v\oplus 0_{k-m}=(v\oplus 0_{k_2-m})\oplus 0_{k_1} \sim_h (w\oplus 0_{k_2-n})\oplus 0_{k_1}=w\oplus 0_{k-n}$$ in $\mathcal{PU}_k(V).$ Now, by Proposition \ref{67}, we get $v\oplus 0_{k-m} \sim_h u\oplus 0_{k-l}$ and $u\oplus 0_{k-l} \sim_h w\oplus 0_{k-n}$ in $\mathcal{PU}_k(V).$ Thus $v\sim_K u$ and $u\sim_K w$ in $\mathcal{PU}_\infty(V).$

\end{proof}

We study some properties of newly defined equivalence relation.

\begin{proposition}\label{71}
Let $V$ be an absolute matrix order unit space and let $u,v,w,u',v'\in \mathcal{PU}_{\infty}(V).$ Then
\begin{enumerate}
\item[(1)] $(u\oplus v)\oplus w=u\oplus (v\oplus w).$
\item[(2)] $u\sim_K u\oplus 0_n$ for all $n\in \mathbb{N}.$
\item[(3)] $u\oplus v\sim_K v\oplus u.$
\item[(4)] If $u\sim_K u'$ and $v\sim_K v',$ then $u\oplus v\sim_K u'\oplus v'.$
\end{enumerate}
\begin{proof}
Let $u \in \mathcal{PU}_m(V).$ For fixed $n \in \mathbb{N},$ we can find $k \in \mathbb{N}$ such that $k > m+n.$ Then by Proposition \ref{67}, we get $u\oplus 0_{k-m} \sim_h u\oplus 0_{k-m}=(u\oplus 0_n)\oplus 0_{k-m-n}$ so that $u\sim_K u\oplus 0_n.$

By Lemma \ref{69}, we have $u\oplus v \sim_h v\oplus u.$ Now, by Corollary \ref{70}(1), we get $u\oplus v \sim_K v\oplus u.$

Next, let $u' \in \mathcal{PU}_l(V),v\in \mathcal{PU}_{m'}(V)$ and $v'\in \mathcal{PU}_{l'}(V)$ such that $u \sim_K u'$ and $v\sim_K v'.$ Then there exist $k_1,k_2 \in \mathbb{N},k_1 > \max \lbrace m,l\rbrace,k_2 > \max \lbrace m',l'\rbrace$ with $u\oplus 0_{k_1-m}\sim_h u'\oplus 0_{k1-l}$ and $v\oplus 0_{k_2-m'}\sim_h v'\oplus 0_{k_2-l'}.$ Put $k=k_1+k_2 > \lbrace m+m',l+l'\rbrace.$ Using part (1), Lemma \ref{69} and Propositions \ref{68}(1) and \ref{67} respectively, we get 
\begin{eqnarray*}
(u\oplus v)\oplus 0_{k-m-m'} &\sim_h & (u\oplus 0_{k_1-m})\oplus (v\oplus 0_{k_2-m'})\\
&\sim_h & (u'\oplus 0_{k_1-l})\oplus (v'\oplus 0_{k_2-l'}) \\
&\sim_h & (u'\oplus v')\oplus 0_{k-l-l'}
\end{eqnarray*}
so that $u\oplus v \sim_K u'\oplus v'.$
\end{proof}
\end{proposition} 

Next, we again define a new relation on partial unitaries by $\sim_K.$

\begin{definition}
Let $V$ be an absolute matrix order unit space. For $u,v\in \mathcal{PU}_\infty(V),$ we say that $u\approx_K v,$ if there exists $w\in \mathcal{PU}_\infty(V)$ such that $u\oplus w\sim_K v\oplus w.$
\end{definition}

The following result shows that $\approx_K$ is stronger than $\sim_K.$

\begin{proposition}\label{72}
Let $V$ be an absolute matrix order unit space and let $n \in \mathbb{N}.$ Then 
\begin{enumerate}
\item[(1)] $\sim_K$ implies $\approx_K$ in $\mathcal{PU}_\infty(V).$
\item[(2)] $\approx_K$ is an equivalence relation on $\mathcal{PU}_\infty(V).$
\end{enumerate}
\end{proposition} 

\begin{proof}
By Proposition \ref{67}, we have $u\oplus w \sim_h u\oplus w$ for all $u,w\in \mathcal{PU}_\infty(V).$ Then by Corollary \ref{70}(1), we get $u\oplus w \sim_K u\oplus w$ so that $u\approx_K u$ for all $u \in \mathcal{PU}_\infty(V).$

Let $u_1,u_2,u_3 \in \mathcal{PU}_\infty(V)$ such that  $u_1 \approx_K u_2$ and $u_2 \approx_K u_3.$ Then $u_1\oplus w_1 \sim_K u_2 \oplus w_1$ and $u_2\oplus w_2\sim_K u_3\oplus w_2$ for some $w_1,w_2 \in \mathcal{PU}_\infty(V).$ By Corollary \ref{70}(2), we have $u_2\oplus w_1\sim_K u_1\oplus w_1$ so that $u_2 \approx_K u_1.$ By Corollary \ref{70} and by Proposition \ref{71}(4), we get $(u_1\oplus w_1)\oplus w_2 \sim_K (u_2 \oplus w_1)\oplus w_2$ and $(u_2\oplus w_2)\oplus w_1\sim_K (u_3\oplus w_2)\oplus w_1.$ Again applying Proposition \ref{71}(3), we have $w_1\oplus w_2 \sim_K w_2 \oplus w_1.$ Then
\begin{eqnarray*}
u_1\oplus (w_1\oplus w_2) &=& (u_1\oplus w_1)\oplus w_2 \sim_K  (u_2\oplus w_1)\oplus w_2 \\
&= & u_2\oplus (w_1\oplus w_2)\sim_K  u_2\oplus (w_2\oplus w_1) \\
&= & (u_2\oplus w_2)\oplus w_1 \sim_K  (u_3\oplus w_2)\oplus w_1 \\
&= & u_3\oplus (w_2\oplus w_1)\sim_K u_3\oplus (w_1\oplus w_2).
\end{eqnarray*} 
Finally, by Corollary \ref{70}(2), we conclude $u_1\oplus (w_1\oplus w_2)\sim_K  u_3\oplus (w_1\oplus w_2).$ Thus $u_1\approx_K u_3.$
\end{proof}

The relation $\approx_K$ enjoys the following properties: 

\begin{corollary}\label{73}
Let $V$ be an absolute matrix order unit space and let $u,v,w,u',v'\in \mathcal{U}_\infty(V)$. Then 
\begin{enumerate}
\item[$1.$] $u\approx_K u\oplus 0_n$ for all $n\in \mathbb{N}.$
\item[$2.$] $u\oplus v\approx_K v\oplus u.$
\item[$3.$] If $u\approx_K u'$ and $v\approx_K v',$ then $u\oplus v\approx_K u'\oplus v'.$
\end{enumerate}
\end{corollary}

\begin{proof}
By Propositions \ref{71} and \ref{72}, it follows that $u\oplus v \approx_K v\oplus u$ and $u\approx_K u\oplus 0_n$ for all $u,v \in \mathcal{U}_\infty(V)$ and $n \in \mathbb{N}.$

Next, let $u,u',v$ and $v' \in \mathcal{PU}_\infty(V)$ such that  $u \approx_K u'$ and $v \approx_K v'.$ Then $u\oplus w_1 \sim_K u' \oplus w_1$ and $v\oplus w_2\sim_K v'\oplus w_2$ for some $w_1,w_2 \in \mathcal{PU}_\infty(V).$ By Proposition \ref{71}, we get 
\begin{eqnarray*}
(u\oplus v)\oplus (w_1\oplus w_2) &=& u\oplus (v\oplus w_1)\oplus w_2 \sim_K  u\oplus (w_1\oplus v)\oplus w_2 \\
&= & (u\oplus w_1)\oplus (v\oplus w_2)\sim_K (u'\oplus w_1)\oplus (v'\oplus w_2)  \\
&= & u'\oplus (w_1\oplus v')\oplus w_2 \sim_K  u'\oplus (v'\oplus w_1)\oplus w_2 \\
&= & (u'\oplus v')\oplus (w_1\oplus w_2).
\end{eqnarray*} 
Thus, by Corollary \ref{70}(2), we conclude that $(u\oplus v)\oplus (w_1\oplus w_2)\sim_K (u'\oplus v')\oplus (w_1\oplus w_2)$ so that $u\oplus v\approx_K u'\oplus v'.$
\end{proof}

Now, we define a binary operatoin in the family of equivalence classes of partial unitaries under $\approx_K.$ 

\begin{proposition}\label{74}
Let $V$ be an absolute matrix order unit space. For each $u,v\in \mathcal{PU}_\infty(V),$ let $[u] = \lbrace w\in \mathcal{PU}_\infty(V): w\approx_K u \rbrace$ and put $[u]+ [v] = [u\oplus v].$ Then, $+$ is a binary operation in the family of equivalence classes $(\mathcal{PU}_\infty(V)/\approx_K,+)$. Also:
\begin{enumerate}
\item[(1)] $[u] + [0] = [u]$ for all $u\in \mathcal{PU}_\infty(V);$
\item[(2)] $[u]+ [v] = [v]+ [u]$ for all $u,v\in \mathcal{PU}_\infty(V);$
\item[(3)] $[u]+ [w] = [v]+ [w]$ for $u,v,w\in \mathcal{PU}_\infty(V),$ then $[u] = [v].$ 
\end{enumerate}
Thus, $(\mathcal{PU}_\infty(V)\big/\approx_K, +)$ is an abelian semi-group satisfying the cancellation law.
\end{proposition}
\begin{proof}
By Corollary \ref{73}(3), it follows that $+$ is well-defined on $\mathcal{PU}_\infty(V)\big/\approx_K.$ Note that (1) and (2) immediately follow from \ref{73}(1) and (2) respectively. Next, we prove (3).

Let $u,v,w\in \mathcal{PU}_\infty(V)$ such that $[u]+ [w] = [v]+ [w].$ Then $u\oplus w \approx_K v\oplus w$ so that $u\oplus (w \oplus x) \sim_K v\oplus (w \oplus x)$ for some $x\in \mathcal{PU}_\infty(V).$ Thus, $u\approx_K v$ so that $[u] = [v].$
\end{proof}

Finally, the following result shows the importance of path homotopy of partial unitaries to define a group.

\begin{theorem}\label{75}
Let $V$ be an absolute matrix order unit space and consider $\mathcal{PU}_\infty(V) \times \mathcal{PU}_\infty(V).$ For all $u_1,u_2,v_1,v_2 \in \mathcal{PU}_\infty(V),$ we define $(u_1,v_1) \equiv_K (u_2,v_2)$ if and only if $u_1\oplus v_2 \approx_K u_2\oplus v_1.$ Then, $\equiv_K$ is an equivalence relation on $\mathcal{PU}_\infty(V) \times \mathcal{PU}_\infty(V).$ Consider: $$K(V)=\lbrace [(u,v)]: u,v\in \mathcal{PU}_\infty(V)\rbrace,$$ where $[(u,v)]$ is the equivalence class of $(u,v)$ in $(\mathcal{PU}_\infty(V) \times \mathcal{PU}_\infty(V),\equiv_K).$ For all $u_1,u_2,v_1,v_2 \in \mathcal{PU}_\infty(V),$ we write: $$[(u_1,v_1)]+ [(u_2,v_2)] = [(u_1\oplus u_2, v_1\oplus v_2)].$$ Then $(K(V),+)$ is an abelian group.
\end{theorem}

\begin{proof}
It follows from Proposition \ref{72}(2) and Corollary \ref{73} that the relation $\equiv_K$ on $\mathcal{PU}_\infty(V) \times \mathcal{PU}_\infty(V)$ is reflexive and symmetric. Let $(u_1,v_1) \equiv_K (u_2,v_2)$ and $(u_2,v_2) \equiv_K (u_3,v_3)$ for some $u_1,u_2,u_3,v_1,v_2,v_3 \in \mathcal{PU}_\infty(V).$ Then $u_1\oplus v_2 \approx_K u_2\oplus v_1$ and $u_2\oplus v_3 \approx_K u_3\oplus v_2.$ By Proposition \ref{72}(2) and by Corollary \ref{73}(2) and (3), we get that $(u_1\oplus v_3)\oplus (u_2 \oplus v_2)\approx_K (u_3\oplus v_1)\oplus (u_2\oplus v_2)$ so that $[u_1\oplus v_3]+ [u_2 \oplus v_2] = [u_3\oplus v_1]+[u_2\oplus v_2].$ Then, by Proposition \ref{74}(3), we conclude that $[u_1\oplus v_3] = [u_3\oplus v_1]$ so that $u_1\oplus v_3 \approx_K u_3\oplus v_1.$ Thus $(u_1,v_1) \equiv_K (u_3,v_3)$ so that $\equiv_K$ is transitive. Hence, $\equiv_K$ is an equivalence relation on $\mathcal{PU}_\infty(V) \times \mathcal{PU}_\infty(V).$

Now, we show that $+$ is well defined on $K(V).$ Let $[(u_1,v_1)]=[(u_1',v_1')]$ and $[(u_2,v_2)]=[(u_2',v_2')].$ Then $(u_1,v_1)\equiv_K(u_1',v_1')$ and $(u_2,v_2)\equiv_K(u_2',v_2')$ so that $u_1\oplus v_1'\approx_K u_1'\oplus v_1$ and  $u_2\oplus v_2'\approx_K u_2'\oplus v_2.$ Again by Proposition \ref{72}(2) and by Corollary \ref{73}(2) and (3), we get that $(u_1\oplus u_2)\oplus (v_1'\oplus v_2')\approx_K (u_1'\oplus u_2')\oplus (v_1\oplus v_2).$ Thus $[(u_1,v_1)]+ [(u_2,v_2)]=[(u_1',v_1')]+ [(u_2',v_2')]$ so that $+$ is well defined.

By Corollary \ref{73}(2) and (3), we have: $$(u_1\oplus u_2)\oplus (v_2\oplus v_1) \approx_K (u_2\oplus u_1)\oplus (v_1\oplus v_2)$$ for all $u_1,u_2,v_1,v_2\in \mathcal{PU}_\infty(V),$ so that $+$ is commutative in $K(V).$

Let $u,v\in \mathcal{PU}_\infty(V).$ By Corollary \ref{73}(1) and (3), we have $(u\oplus 0)\oplus (v\oplus 0) \approx_K u\oplus v.$ Thus $[(0,0)]$ is an identity element in $K(V).$  

Associativity of $+$ on $K(V)$ follows from Proposition \ref{71}(1).

For any $u,v \in \mathcal{PU}_\infty(V),$ we have $[(u\oplus v,v\oplus u)]=[(0,0)].$ Hence $[(v,u)]$ is an inverse element of $[(u,v)]$ in $K(V).$
\end{proof}

\begin{remark}
The abelian group $K(V)$ is called the $K(V)$-group of the absolute matrix order unit space $V.$ If $A$ is a $C^*$-algebra, then $K(A)$ is Grothendieck group of $A$ for partial unitary elements. Thus $K$-group for absolute matrix order unit spaces is a generalization of $K$-group for $C^*$-algebras.
\end{remark}

\subsection{Order structure in $K$}

In this subsection, we prove that $K(V)$ also bears order structure on it. We start with the following discussion:

Let $A$ be a unital $C^*$-algebra. Then 
\begin{enumerate}
\item[(1)] $p\sim_h q$ in $\mathcal{OP}(A)$ implies $p\sim q.$
\item[(2)] $\vert \cdot \vert$ is a continuous function.
\end{enumerate}

For proof of part (1), see \cite[Proposition 2.2.7]{RLL00}. Proof of part (2) can also be found somewhere in literature. However, for the sake of completeness, we provide its proof. For proof of part (2), we use functional calculus of self-adjoint elements in a C$^*$-algebra \cite{RJ83}. It is sufficient to prove that $\sqrt{a_n} \longrightarrow \sqrt{a}$ whenever $a_n,a\in A^+$ such that $a_n \longrightarrow a.$ Let $a_n,a\in A^+$ such that $a_n \longrightarrow a.$ Then there exists $M>0$ such that $\Vert a_n\Vert_n,\Vert a\Vert \leq M$ so that $\sigma(a_n),\sigma(a) \subset [0,M]$ for all $n \in \mathbb{N}.$ Put $f(t)=\sqrt{t}.$ By Weierstrass Approximation theorem, there exists a sequence of polynomials $\lbrace p_m(t)\rbrace$ converging uniformly to $f$ on $[0,M].$ Note that $\Vert p_m(a)-f(a)\Vert,\Vert p_m(a_n)-f(a_n)\Vert \leq \Vert p_m-f\Vert_\infty=\sup\lbrace \vert p_m(t)-f(t)\vert:t\in [0,M]\rbrace.$ Let $\epsilon > 0.$ Choose $m_0 \in \mathbb{N}$ such that $\Vert p_{m_0}-f\Vert_\infty < \frac{\epsilon}{3}.$ Since $m_0$ is fixed and $a_n \longrightarrow a,$ we can also choose $n_0 \in \mathbb{N}$ such that $\Vert p_{m_0}(a_n)-p_{m_0}(a)\Vert < \frac{\epsilon}{3}$ for all $n\geq n_0.$ Then

\begin{eqnarray*}
\Vert \sqrt{a_n}-\sqrt{a}\Vert &=& \Vert f(a_n)-f(a)\Vert \\
&\leq & \Vert f(a_n)-p_{m_0}(a_n)\Vert +\Vert p_{m_0}(a_n)-p_{m_0}(a)\Vert +\Vert p_{m_0}(a)-f(a)\Vert \\
&<& \frac{\epsilon}{3} + \frac{\epsilon}{3} + \frac{\epsilon}{3} \\
&=& \epsilon
\end{eqnarray*}
for all $n\geq n_0.$ Thus $\sqrt{a_n} \longrightarrow \sqrt{a}.$

Next, we prove that $(K(V),K(V)^+)$ is also an ordered abelian group. For that, we need to prove the following result:

\begin{theorem}\label{84}
	Let $V$ be an absolute matrix order unit space. Put $K(V)^+ = \lbrace [(v,0)]: v\in \mathcal{PU}_\infty(V)\rbrace.$ Then 
	\begin{enumerate}
		\item[(1)]$K(V)^+$ is a group cone in $K(V)$. 
		\item[(2)] $K(V)^+$ is generating.
	\end{enumerate} 
	Moreover, if $V$ satisfies: 
	\begin{enumerate}
	\item[(a)] $p\sim_h q$ in $\mathcal{OP}_n(V)$ implies $p\sim q;$  
	\item[(b)] $e^n$ is finite; and 
	\item[(c)] $\vert \cdot \vert_n$ is continuous
	\end{enumerate}
	for all $n\in \mathbb{N}$ and $p,q \in \mathcal{OP}_n(V).$ 
	Then we also have:
	\begin{enumerate}
	\item[(3)]$K(V)^+$ is proper. 
	\end{enumerate}
\end{theorem}

\begin{proof}
	It is routine to verify (1) and (2). Next, we prove (3). Assume that $V$ satisfies $(a),(b)$ and $(c).$ Let $g\in K(V)^+ \cap -K(V)^+.$ There exist $u \in \mathcal{PU}_m(V)$ and $v\in \mathcal{PU}_n(V)$ such that $g=[(u,0)] = [(0,v)]$. Then $(u,0) \equiv (0,v)$ so that $u \oplus v \approx_K 0_m \oplus 0_n$. Thus $u\oplus v \oplus w \sim_h 0_m \oplus 0_n \oplus w$ for some $w\in \mathcal{OP}_l(V).$ By continuity of $\vert \cdot \vert_{m+n+l},$ it follows that $\vert u\vert_m \oplus \vert v\vert_n \oplus \vert w\vert_l \sim_h 0_m \oplus 0_n \oplus \vert w\vert_l.$ Next, by assumption, we have $\vert u\vert_m \oplus \vert v\vert_n \oplus \vert w\vert_l \sim 0_m \oplus 0_n \oplus \vert w\vert_l.$ Since $e^{m+n+l}$ is finte, and $\vert u\vert_m \oplus \vert v\vert_n \oplus \vert w\vert_l$ and $0_m \oplus 0_n \oplus \vert w\vert_l \in \mathcal{OP}_{m+n+l}(V)$ such that $0_m \oplus 0_n \oplus \vert w\vert_l \leq \vert u\vert_m \oplus \vert v\vert_n \oplus \vert w\vert_l,$ by \cite[Corollary 5.1]{PI19}, we conclude that $\vert u\vert_m \oplus \vert v\vert_n \oplus \vert w\vert_l = 0_m \oplus 0_n \oplus \vert w\vert_l$. Then $\vert u\vert_m \oplus \vert v\vert_n =0_m\oplus 0_n$ so that $\vert u\vert_m = 0_m$ and $\vert v\vert_n =0_n.$ By \cite[Proposition 2.4]{PI19}, we have $u=0$ and $v=0.$ Thus $g=0.$ 	
\end{proof}

\begin{corollary}
Let $V$ be an absolute matrix order unit space. Then by Theorem \ref{84}, we conclude that $(K(V),K(V)^+)$ is an ordered abelian group. 
\end{corollary}

\subsection{Functoriality of $K$}

Let $V$ be an absolute matrix order unit space. Then $v\longmapsto[(v,0)]$ defines a map $\Omega_V:\mathcal{PU}_\infty(V)\to K(V)$. In the next result, we describe the functorial nature of $K$. For this, first we prove the following commutative property of $K(V)$:

\begin{theorem}\label{82}
Let $V$ and $W$ be absolute matrix order unit spaces and let $\phi: V \to W$ be a unital completely $\vert \cdot \vert$-preserving map. Then there exists a unique group homomorphism $K(\phi): K(V)\to K(W)$ such that the following diagram commutes: 
\end{theorem}

$$\begin{tikzcd}
\mathcal{PU}_\infty(V)\arrow{r}{\phi}\arrow[swap]{d}{\Omega_V}
& \mathcal{PU}_\infty(W)\arrow{d}{\Omega_W} \\
K(V)\arrow[swap]{r}{K(\phi)}
& K(W)
\end{tikzcd}$$

\begin{proof}
By \cite[Theorem 3.7(1)]{K19}, we get that $\phi(\mathcal{OP}_\infty(V)) \subset \mathcal{OP}_\infty(W)$ so that $\phi(\mathcal{PU}_\infty(V)) \subset \mathcal{PU}_\infty(W)$. Let $v,w \in \mathcal{PU}_{\infty}(V)$ such that $v \sim_h w$. Then $\phi(v) \sim_h \phi(w).$ 

Next, put $K(\phi)([(v,w)])=[(\phi(v),\phi(w))]$ for each $[(v,w)]\in K(V)$. We show that $K(\phi)$ is well-defined. Let $[(v_1,w_1)] = [(v_2,w_2)]$ for some $v_1,v_2,w_1,w_2 \in \mathcal{PU}_\infty(V).$ Then there exists $u\in \mathcal{PU}_\infty(V)$ such that $v_1\oplus w_2 \oplus u \sim_h v_2 \oplus w_1 \oplus u.$ Thus $\phi(v_1)\oplus \phi(w_2) \oplus \phi(u) \sim_h \phi(v_2) \oplus \phi(w_1) \oplus \phi(u)$ so that $[(\phi(v_1),\phi(w_1))] = [(\phi(v_2),\phi(w_2))].$ Hence $K(\phi)$ is well-defined. For all $[(v_1,w_1)],[(v_2,w_2)] \in K(V)$, we have that 

\begin{eqnarray*}
K(\phi)([(v_1,w_1)]+[(v_2,w_2)]) &=& K(\phi)([(v_1 \oplus v_2,w_1 \oplus w_2)]) \\
&=& [(\phi(v_1 \oplus v_2),\phi(w_1 \oplus w_2))] \\
&=& [(\phi(v_1) \oplus \phi(v_2),\phi(w_1) \oplus \phi(w_2))] \\
&=& [(\phi(v_1),\phi(w_1))] + [(\phi(v_2),\phi(w_2))] \\
&=& K(\phi)([(v_1,w_1)]) + K(\phi)([(v_2,w_2)])
\end{eqnarray*}

so that $K(\phi)$ is a group homomorphism. By construction $K$ satisfies the diagram.

\textbf{Uniqueness of $K(\phi)$:-} Let $\mathcal{H}: K(V) \to K(W)$ be a group homomorphism satisfying the same diagram. Then $K(\phi)(\Omega_V(v))=\Omega_V(\phi(v))=\mathcal{H}(\Omega_V(v))$ for all $v \in \mathcal{PU}_\infty(V)$. Thus we get that

\begin{eqnarray*}
K(\phi)([(v,w)]) &=& K(\phi)([(v,0)]-[(w,0)]) \\
&=& K(\phi)(\Omega_V(v)-\Omega_V(w)) \\
&=& K(\phi)(\Omega_V(v)) - K(\phi)(\Omega_V(w)) \\
&=& \mathcal{H}(\Omega_V(v)) - \mathcal{H}(\Omega_V(w)) \\
&=& \mathcal{H}(\Omega_V(v)-\Omega_V(w)) \\
&=& \mathcal{H}([(v,w)])
\end{eqnarray*} 

for all $[(v,w)] \in K(V)$. Hence $K(\phi) = \mathcal{H}$.
\end{proof}

Let $V$ and $W$ be absolute matrix order unit spaces. We denote the zero group homomorphism between $K(V)$ and $K(W)$ by $0_{K(W),K(V)}.$ The identity map on $K(V)$ is denoted by $I_{K(V)}$.

\begin{corollary}\label{83}
Let $U,V$ and $W$ be absolute matrix order unit spaces and let $\phi: U \to V$ and $\psi: V \to W$ be unital completely $\vert \cdot \vert$-preserving maps. Then 
\begin{enumerate}
\item[(a)]$K(I_V)=I_{K(V)};$
\item[(b)]$K(\psi \circ \phi) = K(\psi) \circ K(\phi);$
\item[(c)]$K(0_{W,V}) = 0_{K(W),K(V)}.$
\end{enumerate}
\end{corollary}

\begin{proof}
\begin{enumerate}
\item[(a)] Let $v,w \in \mathcal{PU}_\infty(V)$. Then
\begin{eqnarray*}
K(I_V)([(v,w)]) &=& [(I_V(v),I_V(w))] \\
&=& [(v,w)] 
\end{eqnarray*}

so that by Theorem \ref{82}, $K(I_V)=I_{K(V)}$.
\item[(b)] For any $[(v,w)] \in K(U)$, we get that
\begin{eqnarray*}
K(\psi \circ \phi)([(v,w)]) &=& [(\psi \circ \phi(v),\psi \circ \phi(w))] \\
&=& [(\psi(\phi(v)),\psi(\phi(w)))] \\
&=& K(\psi)[(\phi(v),\phi(w))] \\
&=& K(\psi)(K(\phi)([(v,w)])) \\
&=& K(\psi) \circ K(\phi)([(v,w)]). 
\end{eqnarray*} 
Thus by Theorem \ref{82}, we conclude that $K(\psi \circ \phi) = K(\psi) \circ K(\phi)$.
\item[(c)] $K(0_{W,V})([(v,w)]) = [(0_{W,V}(v),0_{W,V}(w))]=[(0,0)]$ for all $[(v,w)] \in K(V)$. Thus again using \ref{82}, we get that $K(0_{W,V}) = 0_{K(W),K(V)}$.
\end{enumerate}
\end{proof}

\begin{remark}
It follows from Corollary \ref{83} that $K$ is a functor from category of absolute matrix order unit spaces with morphisms as unital completely $\vert \cdot \vert$-preserving maps to category of abelian groups.
\end{remark}

In the end of the subsection, we note the following isomorphic property of $K.$ 

\begin{corollary}
Let $V$ and $W$ be isomorphic absolute matrix order unit spaces (isomorphic in the sense that there exists a unital, bijective completely $\vert \cdot \vert$-preserving map between $V$ and $W$). Then $K(V)$ and $K(W)$ are group isomorphic. 
\end{corollary}

\begin{proof}
Let $\phi:V \to W$ be unital completely $\vert \cdot \vert$-preserving map. Then $\phi^{-1}$ is also unital completely $\vert \cdot \vert$-preserving map. Since $\phi^{-1} \circ \phi = I_{V}$ and $\phi \circ \phi^{-1} = I_{W}$, by Corollary \ref{83}(a) and (b), we get that $K(\phi^{-1}) \circ K(\phi) = I_{K(V)}$ and $K(\phi) \circ K(\phi^{-1}) = I_{K(W)}$. Thus $K(\phi):K(V) \to K(W)$ is a surjective group isomorphism and $K(\phi)^{-1} = K(\phi^{-1})$. Hence $K(V)$ and $K(W)$ are group isomorphic. 
\end{proof}

\section{Relation among $K_0,K_1$ and $K$}

In this section, we derive a relation among $K_0,K_1$ and $K.$ Therefore, we start this section with the following result:

\begin{proposition}\label{76}
Let $(V,e)$ be an absolute matrix order unit space and let $v_i \in \mathcal{PI}_n(V)$ for $i=1,2,3,...,m$ such that $ v_i \perp vj$ for all $i\neq j.$  If $\displaystyle\sum_{i=1}^{m}\vert v_i\vert = e^n = \displaystyle\sum_{i=1}^{m}\vert v_i^*\vert,$ then $\displaystyle\sum_{i=1}^{m} v_i \in \mathcal{U}_n(V).$
\end{proposition}

\begin{proof}
Let $v_i \in \mathcal{PI}_n(V)$ for $i=1,2,3,...,m$ such that $ v_i \perp v_j$ for all $i\neq j.$ Assume that $\displaystyle\sum_{i=1}^{m}\vert v_i\vert = e^n = \displaystyle\sum_{i=1}^{m}\vert v_i^*\vert.$ We prove the above result in the following two cases:

\textbf{Case 1: $m = 2$.} Let $v_1,v_2 \in \mathcal{PI}_n(V)$ such that $v_1 \perp v_2$ and $\vert v_1\vert_n + \vert v_2\vert_n = e^n = \vert v_1^*\vert_n + \vert v_2^*\vert_n.$ By Proposition \ref{88}, we have $\vert v_1 + v_2 \vert_n = \vert v_1\vert_n +\vert v_2\vert_n$ and $\vert v_1^* + v_2^* \vert_n = \vert v_1^*\vert_n + \vert v_2^* \vert_n.$ Thus $\vert v_1+v_2\vert_n=e^n$ and $\vert (v_1+v_2)^*\vert_n=e^n$ so that $v_1+v_2 \in \mathcal{U}_n(V).$ 

\textbf{Case 2: $m > 2$.} Put $w_{m-1}=\displaystyle \sum_{i=1}^{m-1} v_i.$ Thus, by Proposition \ref{89}, we have $w_{m-1}\in \mathcal{PI}_n(V).$ Next, by \cite[Definition 3.4(5)]{K18} and by Remark \ref{92}, we also note that $\vert w_{m-1}\vert_n = \displaystyle\sum_{i=1}^{m-1} \left \vert v_i \right \vert_n \perp \vert v_m\vert_n$ and $\vert w_{m-1}^*\vert_n = \displaystyle\sum_{i=1}^{m-1} \left \vert v_i^* \right \vert_n \perp \vert v_m^*\vert_n$ so that $w_{m-1} \perp v_m.$ Then, again by Proposition \ref{88}, we get $\vert w_{m-1}\vert_n+\vert v_m\vert_n = \displaystyle\sum_{i=1}^{m}\vert v_i\vert = e^n$ and  $\vert w_{m-1}^*\vert_n+\vert v_m^*\vert_n = \displaystyle\sum_{i=1}^{m}\vert v_i^*\vert = e^n.$ Finally, by case (1), we conclude that $\displaystyle\sum_{i=1}^{m} v_i = w_{m-1}+v_m \in \mathcal{U}_n(V).$ 
\end{proof}

The following result describes that every partial unitary is associated to a unitary element.

\begin{lemma}\label{78}
Let $(V,e)$ be an absolute matrix order unit space and let $v\in \mathcal{PU}_n(V)$ for some $n \in \mathbb{N}.$ Then $v \perp e^n-\vert v\vert_n$ and $v\pm (e^n-\vert v\vert_n)\in \mathcal{U}_n(V).$ 
\end{lemma}

\begin{proof}
Let $v \in \mathcal{PU}_n(V).$ Then $\vert v\vert_n = \vert v^*\vert_n \in \mathcal{OP}_n(V)$ so that $\vert v\vert = \vert v^*\vert_n \perp e^n-\vert v\vert_n.$ Thus, by Remark \ref{92}, we get that $v \perp \pm(e^n-\vert v\vert_n).$ Next, $v\pm(e^n- \vert v \vert_n)\in \mathcal{U}_n(V)$ follows from the Proposition \ref{76}.
\end{proof}

Next result provides an alternative characterization of partial unitaries.

\begin{proposition}\label{93}
Let $V$ be an absolute matrix order unit space and let $v \in \mathcal{PI}_n(V)$ for some $n \in \mathbb{N}.$ Then $v \in \mathcal{PU}_n(V)$ if and only if $\vert v^*\vert_n \perp e^n-\vert v\vert_n$ and $v+e^n-\vert v\vert_n \in \mathcal{U}_n(V).$
\end{proposition}

\begin{proof}
Let $n \in \mathbb{N}$ and $v \in \mathcal{PI}_n(V).$ Then $\vert v\vert_n$ and $\vert v^*\vert_n \in \mathcal{OP}_n(V).$ First, assume that $v \in \mathcal{PU}_n(V),$ then $\vert v\vert_n = \vert v^*\vert_n.$ Thus, by Lemma \ref{78}, we conclude that $\vert v^*\vert_n \perp e^n-\vert v\vert_n$ and $v+e^n- \vert v \vert_n \in \mathcal{U}_n(V).$ Conversely, assume that $\vert v^*\vert_n \perp e^n-\vert v \vert_n$ and $v+e^n-\vert v \vert_n \in \mathcal{U}_n(V).$ Then $e^n=\vert (v+e^n-\vert v)^*\vert_n=\vert v^*+(e^n-\vert v\vert_n)\vert_n=\vert v^*\vert_n + e^n-\vert v\vert_n$ so that $\vert v\vert_n = \vert v^*\vert_n.$ Hence $v \in \mathcal{PU}_n(V).$   
\end{proof}

The result stated below tells that converse part of the Lemma \ref{78} also holds.

\begin{corollary}
Let $(V,e)$ be an absolute matrix order unit space and let $n \in \mathbb{N}.$ Then $v\in \mathcal{PU}_n(V)$ if and only if $v \perp e^n-\vert v\vert_n$ and $v\pm(e^n-\vert v\vert_n)\in \mathcal{U}_n(V).$ 
\end{corollary}

\begin{proof}
First, assume that $v \perp e^n-\vert v\vert_n$ and $v\pm(e^n-\vert v\vert_n)\in \mathcal{U}_n(V).$ Since $v\perp e^n-\vert v\vert_n,$ we have $\vert v^*\vert_n \perp e^n-\vert v\vert_n.$ By Proposition \ref{93}, we get that $v\in \mathcal{PU}_n(V).$ Next, converse part immediately follows from the Lemma \ref{78}.
\end{proof}

Next result tells that every partial unitary is arithmetic mean of two unitary elements.

\begin{proposition}\label{77}
Let $(V,e)$ be an absolute matrix order unit space and let $v\in \mathcal{PU}_n(V)$ for some $n \in \mathbb{N}.$ Then, there exist $v_1,v_2\in \mathcal{U}_n(V)$ such that $v=\frac{1}{2}(v_1+v_2)$ and $v_1-v_2 \perp v_1+v_2.$
\end{proposition}

\begin{proof}
Let $n \in \mathbb{N}$ and $v \in \mathcal{PU}_n(V).$ Since $\vert v\vert_n = \vert v^*\vert_n \in \mathcal{OP}_n(V),$ we get that $v$ and $e^n-\vert v\vert_n \in \mathcal{PI}_n(V)$ with $\vert v\vert_n \perp e^n-\vert v\vert_n.$ Put $v_1 = v-e^n+\vert v\vert_n$ and $v_2 = v+e^n-\vert v\vert_n.$ Then, by Proposition \ref{76}, we conclude that $v_1,v_2 \in \mathcal{U}_n(V).$ Note that $v=\frac{1}{2}(v_1+v_2), \vert v_1 -v_2\vert_n = 2 (e^n-\vert v\vert_n),\vert v_1+v_2\vert_n=2\vert v\vert_n.$ By Proposition \ref{90}, we also note that $v_1-v_2 \perp v_1+v_2.$
\end{proof}

In the following result, we prove that for any $n\in \mathbb{N}, \sim_h$ in $\mathcal{PU}_n(V)$ is preserved by $\vert \cdot \vert_n$ in $\mathcal{OP}_n(V)$ and $\mathcal{U}_n(V)$ in terms of corresponding elements under the continuity of $\vert \cdot \vert_n.$

\begin{proposition}\label{87}
Let $(V,e)$ be an absolute matrix order unit space and let $n\in \mathbb{N}.$ If $u\sim_h v$ in $\mathcal{PU}_n(V)$ and $\vert \cdot\vert_n$ is continuous, then $\vert u\vert_n \sim_h \vert v\vert_n $ and $u\pm(e^n-\vert u\vert_n) \sim_h v\pm(e^n-\vert v\vert_n)$ in $\mathcal{OP}_n(V)$ and $\mathcal{U}_n(V)$ respectively. 
\end{proposition}

\begin{proof}
Assume that $u\sim_h v$ in $\mathcal{PU}_n(V)$ and $\vert \cdot\vert_n$ is continuous. Then there exists a continuous function $f:[0,1]\to \mathcal{PU}_n(V)$ such that $f(0)=u$ and $f(1)=v.$ By Definition \ref{86} and by Lemma \ref{78}, we have that $\vert f(t)\vert \in \mathcal{OP}_n(V)$ and $f(t)\pm(e^n-\vert f(t)\vert_n) \in \mathcal{U}_n(V).$ Put $g_1(t)=\vert f(t)\vert_n$ and $g_2^{\pm}(t)=f(t)\pm(e^n-\vert f(t)\vert_n).$ Since $f$ and $\vert \cdot \vert_n$ are continuous, we get that $g_1:[0,1]\to \mathcal{OP}_n(V)$ and $g_2^{\pm}:[0,1]\to \mathcal{U}_n(V)$ are continuous functions such that $g_1(0)=\vert u\vert_n,g_1(1)=\vert v\vert_n,g_2^{\pm}(0)=u\pm(e^n-\vert u\vert_n)$ and $g_2^{\pm}(1)=v\pm(e^n-\vert v\vert_n).$ Thus $\vert u\vert_n \sim_h \vert v\vert_n $ and $u\pm(e^n-\vert u\vert_n) \sim_h v\pm(e^n-\vert v\vert_n)$ in $\mathcal{OP}_n(V)$ and $\mathcal{U}_n(V)$ respectively. 
\end{proof}

Let $(G_1,*_1)$ and $(G_2,*_2)$ be two groups. We write: $G_1\oplus G_2=\lbrace (g_1,g_2):g_1\in G_1,g_2\in G_2 \rbrace$ and $(g_1,g_2)*(h_1,h_2)=(g_1*_1h_1,g_2*_2h_2).$ Then $(G_1\oplus G_2,*)$ is also a group and it is called the direct sum of $G_1$ and $G_2.$ 

Finally, we provides a realtion among $K_0,K_1$ and $K$ by the following result:

\begin{theorem}\label{85}
Let $(V, e)$  be an absolute matrix order unit space satisfying $(T)$ such that $\vert \cdot \vert_n$ is continuous for all $n \in \mathbb{N}$. If $p\sim_h q$ implies $p \sim q$ for all $p,q \in \mathcal{OP}_n(V)$ and $n\in \mathbb{N},$ then there exists a surjective group homomorphism $\theta:K(V)\to K_0(V)\oplus K_1(V).$ In this case, $K(V) \big/ \ker(\theta)\cong K_0(V)\oplus K_1(V).$
\end{theorem}

\begin{proof}
Put $\eta([(u,0)])=[(\vert u \vert_m,0)] \in K_0(V)$ for all $u \in \mathcal{PU}_m(V),m\in \mathbb{N}.$ Let $u,v \in \mathcal{PU}_\infty(V)$ such that $[(u,0)]=[(v,0)].$ Without loss of generality, we may assume that $u,v \in \mathcal{PU}_m(V)$ for some $m \in \mathbb{N}.$ Therefore $u \oplus w\oplus 0_l \sim_h v\oplus w\oplus 0_l$ in $\mathcal{PU}_{m+n+l}(V)$ for some $w \in \mathcal{PU}_n(V)$ and $n,l \in \mathbb{N}.$ Since $\vert \cdot \vert_{m+n+l}$ is continuous, by Proposition \ref{87}, we have that $\vert u \vert_m\oplus \vert w\vert_n\oplus 0_l\sim_h \vert v\vert_m\oplus \vert w \vert_n \oplus 0_l$ in $\mathcal{OP}_{m+n+l}(V).$ By assumption, we get that $\vert u \vert_m\oplus \vert w\vert_n\oplus 0_l\sim \vert v\vert_m\oplus \vert w \vert_n \oplus 0_l$ so that $[(\vert u \vert_m,0)]+[(0,\vert w\vert_n)]=[(\vert v \vert_m,0)]+[(0,\vert w\vert_n)].$ By cancellation law in $K_0(V),$ we get that $[(\vert u \vert_m,0)]=[(\vert v \vert_m,0)].$ Thus, $\eta$ is well-defined. Next, we have $[(u,v)]=[(u,0)]-[(v,0)]$ for all $u,v \in \mathcal{PU}_\infty(V).$ Define $\eta([(u,v)])=\eta([(u,0)])-\eta([(v,0)])$ for all $u,v \in \mathcal{PU}_\infty(V).$ Then, $\eta$ extends to $K(V).$ It is routine to verify that $\eta: K(V) \to K_0(V)$ is a group homomorphism. By Lemma \ref{78}, it follows that $u+e^m-\vert u\vert_m \in \mathcal{U}_m(V).$ Also, put $\mu([(u,0)])=[(u+e^m-\vert u \vert_m,e)] \in K_1(V)$ for all $u \in \mathcal{PU}_m(V).$ Again, by Proposition \ref{87}, we note that $(u+e^m-\vert u\vert_m) \oplus (w+e^n-\vert w\vert_n)\oplus e^l = (u\oplus w\oplus 0_l)+e^{m+n+l}-(\vert u\vert_m \oplus \vert w\vert_n \oplus 0_l)\sim_h (v\oplus w\oplus 0_l)+e^{m+n+l}-(\vert v\vert_m \oplus \vert w\vert_n \oplus 0_l)=(v+e^m-\vert w\vert_m) \oplus (v+e^n-\vert w\vert_n)\oplus e^l.$ We can choose $l> \max\lbrace m,n\rbrace$ so that $(u+e^m-\vert u\vert_m) \oplus (w+e^n-\vert w\vert_n)\sim_1 (v+e^m-\vert w\vert_m) \oplus (v+e^n-\vert w\vert_n)$ and consequently $(u+e^m-\vert u\vert_m)\approx_1 (v+e^m-\vert w\vert_m).$ Then $[(u+e^m-\vert u\vert_m,e)]=[(v+e^m-\vert w\vert_m,e)].$ Thus, $\mu$ is also well-defined. Define $\mu([(u,v)])=\mu([(u,0)])-\mu([(v,0)])$ for all $u,v \in \mathcal{PU}_\infty(V).$ It is also routine to verify that $\mu: K(V)\to K_1(V)$ is a group homomorphism. 

Now, define a new group homomorphism $\theta:K(V)\to K_0(V)\oplus K_1(V)$ by $\theta([(u,0)])=(\eta([(u,0)]),\mu([(u,0)]))$ for all $u\in \mathcal{PU}_\infty(V).$ We claim that $\theta$ is surjective. Let $p\in \mathcal{OP}_m(V)$ and $v\in \mathcal{U}_n(V)$ for some $m,n\in \mathbb{N}.$ Without loss of generality, we can assume that $m=n.$ Put $p'=e^n-p.$ Then $\theta([(p',0)])=([(p',0)],[(e^n,e)])$ and $\theta([(v,0)])=([(e^n,0)],[(v,e)]).$ By Proposition \ref{15}(5), we get that $e^n \sim p\oplus p'$ and hence $[(e^n,0)]=[(p,0)]+[(p',0)].$ Thus, we have 
\begin{eqnarray*}
\theta([(v,0)]-[(p',0)]) &=& \theta([(v,0)])-\theta([(p',0)])\\
&=& ([(e^n,0)],[(v,e)])-([(p',0)],[(e^n,e)])\\
&=& ([(e^n,0)]-[(p',0)],[(v,e)]-[(e^n,e)])\\
&=& ([(p,0)],[(v\oplus e,e\oplus e^n)])\\
&=& ([(p,0)],[(v,e)])
\end{eqnarray*}

so that $\theta$ is surjective. Finally, by first isomorphism theorem of groups, we conclude that $K(V) \big/ \ker(\theta)\cong K_0(V)\oplus K_1(V).$
\end{proof}

\begin{remark}
If $\theta$ is injective, then $K(V)\cong K_0(V)\oplus K_1(V).$
\end{remark}


\begin{thebibliography}{100}
	
\bibitem{MA71} E. M. Alfsen, Compact Convex sets and Boundary Integrals, Springer-Verlag, Berlin-Heidelberg-New York, 1971. 

\bibitem{LA} L.\  Asimov, Universally well-capped cones, Pacific J. Math., {\bf 26} (1968), 421--431.

\bibitem{LA73} L. Asimov, Monotone extension in ordered Banach spaces and their duals, J. Lond. Math. Soc., {\bf 6} (1973), 563--569.

\bibitem{A74} L. Asimov, Complementary cones in dual Banach Spaces, Illinois J. Math., {\bf 18} (1974), 657--668.

\bibitem{B98} B. Blackadar, K-Theory for operator algebras, Cambridge University Press, Cambridge, 1998.
	
\bibitem{B06} B. Blackadar, Operator algebras. Theory of C$^*$-algebras and von Neumann algebras, Springer-Verlag, Berlin- Heidelberg-New York, 2006.	

\bibitem{B54} F. F. Bonsall, Sublinear functionals and ideals in partially ordered vector spaces, Proc. Lond. Math. Soc., {\bf 4} (1954), 402--418.

\bibitem{B56} F. F. Bonsall, Regular ideals of partially ordered vector spaces, Proc. Lond. Math. Soc., {\bf 6} (1956), 626--640.
	
\bibitem{CE77} M. D. Choi and E. G. Effros, Injectivity and operator spaces, J. Funct. Anal., {\bf 24} (1977), 156--209.

\bibitem{E64} D. A. Edwards, On the homeomorphic affine embedding of a locally compact cone into a Banach dual space endowed with the vague topology, Proc. Lond. Math. Soc., {\bf 14} (1964), 399--414.

\bibitem{ER88} E. G. Effros and Z. J. Ruan, On matricially normed spaces, Pacific J. Math., {\bf 132} (1988), 243--264.
	
\bibitem{JE64} A. J. Ellis, The duality of partially ordered normed linear spaces, J. Lond. Math. Soc., {\bf 39} (1964), 730--744.

\bibitem{GN} I. Gelfand and M. Neumark, On the imbedding of normed rings into the ring of operators in Hilbert space, Rec. Math. [Mat. Sbornik] N.S., {\bf 54} (1943), 197--213. 

\bibitem{KRG86} K. R. Goodearl, Partially ordered abelian groups with interpolation, Mathematical Surveys and Monographs, Amer. Math. Soc., {\bf 20} (1986).

\bibitem {J72} G. Jameson, Ordered linear spaces, Lecture Notes in mathematics, Springer-Verlag, Berlin-Heidelberg-New York, {\bf 141} (1970).

\bibitem{RVK} R.\ V.\ Kadison, A representation theory for commutative topological algebra, Mem. Amer. Math. Soc., {\bf 7} (1951). 

\bibitem{Kad51} R. V. Kadison, Order properties of bounded self-adjoint operators, Proc. Amer. Math. Soc., {\bf 2} (1951), 505--510. 

\bibitem{RJ83} R. V. Kadison and J. R. Ringrose, Fundamentals of the theory of operator algebras, Academic Press, Inc., London-New York,  1983. 

\bibitem{Kak} S. Kakutani, Concrete representation of abstract $M$-spaces, Ann. of Math., {\bf 42} (1941), 994--1024.

\bibitem{K10} A. K. Karn, A p-theory of ordered normed spaces, Positivity, {\bf 14} (2010), 441--458. 

\bibitem{K14} A. K. Karn, Orthogonality in $l_p$-spaces and its bearing on ordered Banach spaces, Positivity, {\bf 18} (2014), 223--234. 

\bibitem{K16} A. K. Karn, Orthogonality in C$^*$-algebras, Positivity, {\bf 20} (2016), 607--620.

\bibitem{K18} A. K. Karn, Algebraic orthogonality and commuting projections in operator algebras, Acta Sci. Math. (Szeged), {\bf 84} (2018), 323--353. 

\bibitem{K19} A. K. Karn and A. Kumar, Isometries of absolute order unit spaces, Positivity, {\bf 24} (2020), 1263--1277.

\bibitem{PI19} A. K. Karn and A. Kumar, Partial isometries in an absolute order unit space, Banach J. Math. Anal., {\bf 15} (2021), 1--26. 

\bibitem{KO21} A. K. Karn and A. Kumar, $K_0$-group of absolute matrix order unit spaces, Adv. Oper. Theory, {\bf 6} (2021), 1--27.

\bibitem{KV97} A. K. Karn and R. Vasudevan, Approximate matrix order unit spaces, Yokohama Math. J., {\bf 44} (1997), 73--91.
 
\bibitem{KFN} K.\ F.\ Ng, The duality of partially ordered Banach spaces, Proc. Lond. Math. Soc, {\bf 19} (1969), 269--288.

\bibitem{GKP} G. K. Pedersen, C$^*$-algebras and their automorphism groups, Academic Press, Inc., London-New York, 1979.

\bibitem{RLL00} M. R\o{}rdam, F. Larsen $\&$ N.J. Laustsen, An Introduction to K-theory for C$^*$-Algebras, Cambridge University Press, Cambridge, 2000.

\bibitem{ZJR88} Z. J. Ruan, Subspaces of C$^*$-algebras, J. Funct. Anal., {\bf 76} (1988), 217--230.

\bibitem{WO93} N. E. Wegge-Olsen, K-Theory and C$^*$-algebras, Oxford University Press, New York, 1993. 
	
\bibitem {WN73} Y. C. Wong and K. F. Ng, Partially ordered topological vector spaces, Oxford Mathematical Monographs, Clarendon Press, Oxford, 1973.	
	 
\end{thebibliography}
\end{document}